\documentclass[reqno]{amsart}
\usepackage{graphicx}
\usepackage{amssymb}
\usepackage{mathtools}
\usepackage{epstopdf}
\usepackage{nicefrac}
\usepackage{amsthm}
\usepackage{url}
\usepackage{graphicx}
\usepackage{fancybox}
\usepackage{tikz}
\usepackage{bbm}
\usepackage{shuffle}
\usepackage{color}
\usepackage{tikz-cd}
\usetikzlibrary{arrows}
\theoremstyle{plain}

\newtheorem{thm}{Theorem}
\newtheorem*{thmA}{Theorem A}
\newtheorem*{thmB}{Theorem B}

\newtheorem*{cor}{Corollary}

\newtheorem{lemma}{Lemma}

\newtheorem{prop}{Proposition}

\theoremstyle{remark}

\newcommand{\e}{{\textbf{e}}}

  \def\a{\alpha} \def\b{\beta}  
   \def\C{\mathbb{C}}  \def\Z{\mathbb{Z}} \def\N{\mathbb{N}} \def\Q{\mathbb{Q}} 
\def\a{\alpha} \def\z{\zeta} \def\g{\gamma} \def\o{\omega} \def\th{\theta} \def\l{\ell} \def\t{\tau}     
\def\={\;=\;} \def\X{\textbf{S}} \def\G{\Gamma} \def\h{\frac12} \def\Vt{\widetilde V} \def\O{\text O}
\def\MZV{multiple zeta value}  \def\MZVs{\MZV s}

 \def\bz{\textbf{z}} \def\bs{\textbf{s}}

  \def\AC{3}

\begin{document}

\rightline{IPHT-t19/069}

\vspace{0.3 cm}

\title[Genus-zero and genus-one string amplitudes and special multiple zeta values]{Genus-zero and genus-one string amplitudes \\ and special multiple zeta values}

 
\author{Don Zagier}
\address{Max-Planck-Institut f\"ur Mathematik, Vivatsgasse 7, 53111 Bonn (Germany)}
\email{dbz@mpim-bonn.mpg.de}
\author{Federico Zerbini}
\address{Institut de Physique Th\'eorique (IPhT), CEA-Saclay, Orme des Merisiers, B\^atiment 774, 91191 Gif sur Yvette (France)}
\email{federico.zerbini@ipht.fr}
 

\maketitle
\date{today}

\section*{Abstract}

\noindent In this paper we show that in perturbative string theory the genus-one contribution to formal 2-point amplitudes can be related to the genus-zero contribution to 4-point amplitudes. This is achieved by studying special linear combinations of multiple zeta values that appear as coefficients of the amplitudes. We also exploit our results to relate closed strings to open strings at genus one using Brown's single-valued projection, proving a conjecture of~\cite{BSZ}.

\section{Introduction}\label{Sec:Intro}

\noindent Define two meromorphic functions $V^{({\rm op})}$ and $V^{({\rm cl})}$ on $\,\X:=\,\{(s,t,u)\in\C^3\mid s+t+u=0\}\,$ by
\begin{equation}\label{AA}
V^{({\rm op})}(s,t,u) \= \frac{\G(1+s)\,\G(1+t)}{\G(1-u)}\,,\quad\,\,\, V^{({\rm cl})}(s,t,u) \= \frac{\G(1+s)\,\G(1+t)\,\G(1+u)}{\G(1-s)\,\G(1-t)\,\G(1-u)}\,.
\end{equation}

The first of these is a variant of Veneziano's open bosonic string amplitude~\cite{Veneziano}, while the second is a variant of Virasoro's closed bosonic string amplitude~\cite{Virasoro}. Their study in the context of scattering amplitudes in the late 1960's marked the birth of string theory. The letter ``V'' is a reminder of the two originators, while the superscripts ``(op)'' and ``(cl)'' stand for open and closed strings, respectively. In Section \ref{Sec:Virasoro} we will study the Taylor expansion of $V^{({\rm cl})}(s,t,u)$. Euler's formula for $\log\G(1+x)$ gives
\begin{equation}\label{ExpVirasoro1}
V^{({\rm cl})}(s,t,u)\=\exp\bigg(-2\,\sum_{n\geq 1}\,\frac{\zeta(2n+1)}{2n+1}\,(s^{2n+1}\,+\,t^{2n+1}\,+\,u^{2n+1})\bigg),
\end{equation}
so that each Taylor coefficient of $V^{({\rm cl})}(s,t,u)$ is a polynomial with rational coefficients in the odd zeta values $\z(3)$, $\z(5)$, $\z(7),\,\dots\,$ (see also Proposition~\ref{Proposition 1}).
We will show in Theorem~\ref{Theorem 1} that these Taylor coefficients also belong to the $\mathbb{Q}$-span of the special linear combinations of \MZVs\footnote{Our convention for multiple zeta values is that $\zeta(k_1,\ldots ,k_r)=\sum_{0<n_1<\cdots <n_r}n_1^{-k_1}\cdots n_r^{-k_r}$.}\ defined for $k\geq 2$ and $r\geq 0$ by
\begin{equation}\label{Z}
Z(k,r)\;:=\;\sum_{\substack{\textbf{r}\in\{1,2\}^j,\, j\geq 0\\r_1+\cdots +r_j=r}}2^{\#\{i\,:\,r_i=1\}}\,\zeta(\textbf{r},k),
\end{equation}
the first cases of which are given by
\begin{align*}
 Z(k,0)&\=\z(k)\,, \notag\\
 Z(k,1)&\=2\,\z(1,k)\,, \notag\\
 Z(k,2)&\=\z(2,k)\,+\,4\,\z(1,1,k)\,, \notag\\
 Z(k,3)&\=2\,\z(1,2,k)\,+\,2\,\z(2,1,k)\,+\,8\,\z(1,1,1,k)\,, \notag\\
 Z(k,4)&\=\z(2,2,k)\,+\,4\,\z(1,1,2,k)\,+\,4\,\z(1,2,1,k)\,+\,4\,\z(2,1,1,k)\,+\,16\,\z(1,1,1,1,k)\,.
\end{align*}

This is thus the reverse of what one usually tries to do
in the theory of \MZVs, which is to try to express $\Q$-linear combinations of \MZVs\ as polynomials in the 
more familiar single zeta values.  Here the process is justified for two reasons. On the one hand, the 
expression in terms of \MZVs\ is much more compact than the one in terms of products of single
zeta values (namely, $\,\O(q)\,$ terms rather than $\,\O(p^{q-1})\,$ terms for the coefficients $e_{p,q}$
defined below).  More interestingly, the specific linear combinations of \MZVs\ $Z(k,r)$ that occur turn out to be the same ones that had occurred in an earlier calculation
of Green, Russo and Vanhove~\cite{GRV} in the leading term of certain non-holomorphic
modular functions in the upper half-plane~$D_\l(\tau)$, known in the literature as (two-point) \emph{modular graph functions}, which contribute to the computation of genus-one closed superstring amplitudes. More specifically, $D_\l(\tau)=d_\l(Y)+O(e^{-Y})$ for $Y:=2\pi\,\text{Im}(\tau)\rightarrow +\infty$ and~$d_\l(Y)$ a Laurent polynomial, and the results of the paper~\cite{GRV} and its appendix showed that the coefficients of $d_\l(Y)$ are rational linear combinations of the numbers $Z(k,r)$ (see Proposition~\ref{TeoGRV}). Numerical calculations from~\cite{GRV} suggested the following theorem, which will be proved in Section~\ref{SSec:1stProof}:
\begin{thmA}
The coefficients of the Laurent polynomial $d_\l(Y)$ are polynomials in odd zeta values with rational coefficients.
\end{thmA}
We will actually prove this in two different ways, with each proof providing a stronger result, as they both relate the coefficients of~$d_\l(Y)$ to the Taylor coefficients~$e_{p,q}$ of the Virasoro function~$V^{({\rm cl})}$. A third and different proof of Theorem A was found very recently by D'Hoker and Green~\cite{DG19}. Our first proof exploits the above-mentioned fact that both the coefficients of~$d_\l(Y)$ and the numbers~$e_{p,q}$ belong to the rational vector spaces spanned by the special \MZVs \ ~$Z(k,r)$. Our second proof does not use the results from Section~\ref{Sec:Virasoro} but instead directly relates the coefficients of~$d_\l(Y)$ to the Virasoro function~$V^{({\rm cl})}$. Here we indicate how the latter approach gives a link between open and closed string amplitudes at genus one, which is the topic of the second part of the paper. Specifically, the coefficients of the Laurent polynomials~$d_\l(Y)$ are given by simple linear combinations~$\gamma_{n,k}$ of special \MZVs \ (see Proposition~\ref{Proposition 4}), whose generating series
\begin{equation*}
W^{({\rm cl})}(X,Y)\;:=\;\frac{1}{X\,(X+Y)\,(Y-X)}\;+\;\sum_{n>k>0}\gamma_{n,k}\,X^{n-k-1}\,Y^{2k-2}
\end{equation*}
is shown to  satisfy (Proposition~\ref{ThmMainClosed}) the identity
\begin{equation}\label{IntroWcl}
W^{({\rm cl})}(X,Y)\=\frac{V^{({\rm cl})}(2X,-X-Y,Y-X)}{X\,(X+Y)\,(Y-X)}\,,
\end{equation}
which in particular implies Theorem A. The superscript ``(cl)'' refers once again to the fact that we are considering numbers coming from closed strings.

We remark that a purely number-theoretical consequence of the results of this first part of the paper is to highlight two special subspaces (see Section~\ref{SSec:epq=0})
\begin{equation}\label{calEcalZ1}
\mathcal{E}_w\,:=\,\big{\langle} \,e_{p,q}\;\big{|}\;p\geq 0,\,q\geq 1,\, 2p+3q=w\,\big{\rangle}_{\mathbb{Q}} \quad \subseteq \quad \mathcal{D}_w\,:=\,\big{\langle} \,Z(w-r,r)\;\big{|} \;0\leq r\leq w-2\,\big{\rangle}_{\mathbb{Q}},
\end{equation}
both of dimension~$O(w)$, the first spanned by the Taylor coefficients~$e_{p,q}$ of~$V^{({\rm cl})}$ (see formula~(\ref{AD})) or equivalently by the coefficients~$\gamma_{n,k}$ of the Laurent polynomials~$d_\l(Y)$, and the second spanned by the numbers $Z(k,r)$ from~(\ref{Z}). These vector spaces are contained in the presumably much larger spaces~$\mathcal{R}_w$ and~$\mathcal{Z}_w$ of weight-$w$ polynomials in odd zeta values and \MZVs, respectively\footnote{The dimensions of $\mathcal{R}_w$ and $\mathcal{Z}_w$ conjecturally grow roughly like $e^{\pi\sqrt{w/3}}$ and
$(1.3247\cdots)^w$, respectively, although neither dimension can be proved to be larger than~1 for any~$w$.}, and have interesting structures, studied in Section~\ref{SSec:epq=0}.

We now turn to the open string side of this story. Holomorphic (but not quite modular\footnote{The functions~$B_\l(\tau)$ can be written as special combinations of iterated Eichler integrals of Eisenstein series~\cite{MMV} called \emph{elliptic multiple zeta values}~\cite{Enriquez}.}) analogues~$B_\l(\tau)$ of the functions~$D_\l(\tau)$ were recently introduced in~\cite{BSZ} and related to the computation of genus-one open superstring amplitudes. They are known as (two-point, B-cycle) \emph{holomorphic graph functions}. It was shown in~\cite{BSZ} that $B_\l(\tau)=b_\l(T)+O(e^{-T})$ for $T:=\pi\tau/i\rightarrow +\infty$ and~$b_\l(T)$ a Laurent polynomial with coefficients contained in the ring of \MZVs. In Section~\ref{Sec:Open} we study these Laurent polynomials and we relate them to the genus-zero Veneziano function~$V^{({\rm op})}$. More specifically, the coefficients of~$b_\l(T)$ are simple linear combinations of special \MZVs \  ~$\eta_{n,k}$ (see Theorem~\ref{PropOp} and its corollary) and we will show (Proposition~\ref{propWop}) the identity
\begin{equation}\label{IntroWop}
W^{({\rm op})}(X,Y)\=\bigg(\frac{\sin\big(\pi(X+Y)\big)}{\sin\big(\pi(Y-X)\big)}\,-\,1\bigg)\,\frac{V^{({\rm op})}(2X,-X-Y,Y-X)}{\,2\,X^2\,(X+Y)}\,,
\end{equation}
where $W^{({\rm op})}(X,Y)$ is the generating series of the numbers~$\eta_{n,k}$
\begin{equation*}
W^{({\rm op})}(X,Y)\;:=\;\frac{1}{X\,(X+Y)\,(Y-X)}\;+\;\sum_{n\geq k>0}\eta_{n,k}\,X^{n-k-1}\,Y^{2k-2}.
\end{equation*}

For $V^{({\rm op})}$, as opposed to $V^{(\rm cl)}$, Euler's expansion of $\log\G(1+x)$ gives
\begin{equation}\label{ExpVeneziano1}
V^{({\rm op})}(s,t,u)\=\exp\bigg(\sum_{n\geq 2}\frac{(-1)^n\zeta(n)}{n}\big(s^n+t^n-(-u)^n\big)\bigg),
\end{equation}
implying, together with equation~(\ref{IntroWop}), that also the open string coefficients $\eta_{n,k}$ can be reduced to single zeta values (Theorem~\ref{PropOp}), but this time involving also even zeta values $\zeta(2k)$. Far from being an accident, this is part of a very intriguing pattern which seems to relate open strings to closed strings by just mapping all periods appearing in the open case to the corresponding \emph{single-valued periods}\footnote{More precisely, such a map can be defined for special motivic periods, including motivic multiple zeta values, and it induces a well-defined map on the actual periods only if we assume the period conjecture.}. The latter are given by a single-valued integration pairing between differential forms and dual differential forms, defined by Brown~\cite{BrownSVMZV} by transporting the action of complex conjugation from singular to de Rham cohomology via the comparison isomorphism. A simple special case of this ``single-valued projection'' gives for all $k\geq 1$
\begin{equation}\label{singlevaluedZetas}
{\rm sv}(\zeta(2k))\=0\,, \quad \quad \quad \quad \quad {\rm sv}(\zeta(2k+1))\=2\zeta(2k+1).
\end{equation}
In particular, extending this by linearity to formal power series with rational coefficients and comparing~(\ref{ExpVeneziano1}) with~(\ref{ExpVirasoro1}) we have ${\rm sv}\big(V^{({\rm op})}(s,t,u)\big)=V^{({\rm cl})}(s,t,u)$. Combining this with equations~(\ref{IntroWcl}) and~(\ref{IntroWop}), we prove in Section~\ref{Ssec:ProofEsvConj} the identity ${\rm sv}\big(W^{({\rm op})}(X,Y)\big)=W^{({\rm cl})}(X,Y)$,
or equivalently ${\rm sv}(\eta_{n,k})=\gamma_{n,k}$, which implies our second main result of this article:
\begin{thmB}  
For all $l\geq 1$ we have ${\rm sv}\big(b_\l(X)\big)=d_\l(X)$.
\end{thmB} 
This is a special case of a general conjecture that appeared in~\cite{BSZ} relating open and closed string amplitudes at genus one via Brown's single-valued map.

Equations~(\ref{IntroWcl}) and~(\ref{IntroWop}) yield also analogues of the KLT formula\footnote{This formula was discovered by Kawai, Lewellen and Tye in~\cite{KLT}. It expresses $n$-point genus-zero closed string amplitudes as ``double copies'' of genus-zero open string amplitudes.} for the 2-point graph functions appearing in genus-one superstring amplitudes (see Section~\ref{Ssec:KLT}). Finding higher-genus analogues of the KLT formula is one of the main open problems in this area, and we hope that our result may point towards a genus-one generalization.

We conclude this introduction by indicating briefly what is the physical relevance of our results. 
The two functions in~(\ref{AA}) occur in the lowest-order (i.e.\ genus zero) 
approximations to 4-gluon and 4-graviton massless state scattering in superstring theory, respectively~\cite{GSW}. In this context, the variables $s$, $t$ and $u$, called \emph{Mandelstam variables}, are proportional to the scalar products of 
pairs of the momentum vectors of the four scattering particles. Because these four momenta have  
Minkowski norm~0 (the particles are massless and on-shell) and sum to~0 (momentum conservation), there are only
three distinct scalar products, with sum~0. The proportionality constant is a rational multiple of the inverse string tension $\alpha'$, which makes the Mandelstam variables dimensionless. By contrast, the functions $B_\l(\tau)$ and $D_\l(\tau)$ concern 2-gluon and 2-graviton genus-one superstring amplitudes, respectively. They do not appear to have a physical meaning by themselves, because they would depend on just one identically vanishing Mandelstam variable, but they contribute to the 4-point genus-one amplitude~\cite{BGS, BSZ, GRV}.

The possible physical interest for the results obtained in this paper is twofold. On the one hand, we have connected string amplitudes of different genera, a phenomenon which is possibly related to a physical prediction called ``unitarity'' as well as to the fact that, in the limit $\rm{Im}(\tau)\rightarrow i\infty$, the torus~$\mathcal{E}_\tau$ with $n$ marked points can be thought of as a sphere with $n+2$ marked points, two of which are colliding. On the other hand, Theorem~B gives a first genus-one confirmation, after the recent proofs at genus zero for any number of points~\cite{BrownDupont, SchlSchn, VanZerb18}, that closed string amplitudes may be obtained from open string amplitudes using Brown's construction of single-valued periods. It would be interesting to see if this can be explained as yet another ``double-copy relation'' between gauge theories and gravity.
We can summarize how our results fit into the string theory context with the following diagram\footnote{We denote by $A_{g,n}^{(\cdot)}$ the $n$-point genus-$g$ configuration-space contribution to the superstring amplitude, and by writing $n=2$ we formally interpret the $n=2$ contribution to the $n=4$ amplitude as a 2-point amplitude.}:
\[
\begin{tikzcd}
A_{0,4}^{({\rm op})}(s,t) \arrow[r, "{\rm sv}"]
& A_{0,4}^{({\rm cl})}(s,t) \\
A_{1,2}^{({\rm op})}(s,\tau) \arrow[r, "{\rm sv}"] \arrow[u, "{\rm \tau\rightarrow i\infty}"]
& A_{1,2}^{({\rm cl})}(s,\tau) \arrow[u, "{\rm \tau\rightarrow i\infty}"]
\end{tikzcd}
\]

\section{The Virasoro function and special multiple zeta values}\label{Sec:Virasoro}

\noindent As already mentioned, the Virasoro function~$V^{({\rm cl})}$ from~(\ref{AA}) appears in the computation of the genus-zero 4-graviton scattering amplitude in Type~II superstring theory. More specifically, the latter is essentially given by an integral over the complex projective line\footnote{The perturbative expansion of $n$-point closed superstring amplitudes is a series whose $g$-th coefficient is an integral over the moduli space of~$n$-punctured super Riemann surfaces of genus~$g$. Such integrals, in all cases considered in this article, reduce to integrals over the moduli space of ordinary Riemann surfaces $\mathfrak{M}_{g,n}$~\cite{DonagiWitten}. Therefore if $g=0$ and $n=4$ the integration domain is $\mathbb{P}^1_\mathbb{C}$.}, known as the \emph{complex beta integral}, of the form
\begin{equation}\label{ComplexBeta}
\beta_{\mathbb{C}}(s,t)\=-\frac{1}{2\pi i}\int_{\mathbb{P}^1_\mathbb{C}} \,|z|^{2s-2}\,|1-z|^{2t-2}\,dz\,d\overline{z},
\end{equation}
which converges absolutely for $\text{Re}(s),\text{Re}(t)>0$ and $\text{Re}(s+t)<1$. It is related to the Virasoro function by the equation
\begin{equation}\label{betaShapiro}
\beta_{\mathbb{C}}(s,t)\=-\frac{u}{st}\,\,V^{({\rm cl})}(s,t,u),
\end{equation}
where we pass from the minimal set of independent Mandelstam variables $s$ and $t$ to the more symmetric set of variables $(s,t,u)\in\X$, with $\X$ as in the introduction. Indeed, using polar coordinates, the binomial theorem, the substitution $r^2=v/(1-v)$,
the usual beta integral, the reflection formula for $\G(x)$, and Gauss's formula for $_2F_1(a,b;c;1)$ we can write
\begin{align*}
\beta_{\mathbb{C}}(s,t) &\= \frac1\pi \int_0^\infty r^{2s-1}\,\int_0^{2\pi}(r^2-2r\cos\th+1)^{t-1}\,d\th\;dr \notag\\
 &\= \sum_{n=0}^\infty\binom{t-1}{2n}\,\cdot\,2\int_0^\infty r^{2s+2n-1}(r^2+1)^{t-2n-1}dr\,\cdot\,
   \frac1{2\pi}\,\int_0^{2\pi}\bigl(e^{i\th}+e^{-i\th})^{2n}\,d\th \notag\\
 &\=\sum_{n=0}^\infty\,\frac{\G(t)}{(2n)!\,\G(t-2n)}\; \frac{\G(s+n)\,\G(1+u+n)}{\G(2n-t+1)}\;\binom{2n}n 
 \= \frac{\G(s)\,\G(1+u)}{\G(1-t)}\,_2F_1(s,1+u;1;1)\notag \\
 &\= \frac{\G(s)\,\G(t)\,\G(1+u)}{\G(1-s)\,\G(1-t)\,\G(-u)}\=-\frac{u}{st}\,\,\frac{\G(1+s)\,\G(1+t)\,\G(1+u)}{\G(1-s)\,\G(1-t)\,\G(1-u)}\;.  
\end{align*}
In particular the complex beta integral, like the real one, is a quotient of products of gamma functions. This computation, familiar to physicists but less so to mathematicians, is due to Shapiro~\cite{Shapiro}. Studying the Taylor expansion at the origin of $V^{({\rm cl})}$ is therefore equivalent, from the physical viewpoint, to studying the ``$\alpha'$-expansion'' as the inverse string tension $\alpha'\rightarrow 0$, also known as ``low-energy expansion'', whose coefficients give the higher order string theory corrections to the corresponding field theory (which is, in this case, supergravity).

\subsection{The Taylor expansion of $V^{({\rm cl})}$}\label{Ssec:TaylorVira}

The function defined by~(\ref{AA}) has the following properties: 
\begin{itemize}
\item $\,V^{({\rm cl})}(s,t,u)\,$ is symmetric in $s$, $t$ and $u$\,;
\item $\,V^{({\rm cl})}(s,t,u)=1\,$ if $\,stu=0\,$;
\item $\,V^{({\rm cl})}(s+1,t-1,u)=\dfrac{s(s+1)}{t(t-1)}\,V^{({\rm cl})}(s,t,u)\,$;
\item $\,V^{({\rm cl})}(s,t,u)V^{({\rm cl})}(-s,-t,-u)=1\,$;
\item $\,V^{({\rm cl})}(s,t,u)\,$ has simple poles along the lines $\,s=-n$, $t=-n$, $u=-n\,$ with $n\in\N$,
simple zeros along the lines $\,s=n$, $t=n$, $u=n\,$ with $n\in\N$, and no other zeros or poles.
\end{itemize}
All of these properties follow immediately from the definition.  (For the second, note that if, for instance,
$u=0$, then $1+u=1-u$ and $1\pm s=1\mp t$.) The first one implies that $V^{({\rm cl})}(s,t,u)$ can be
written as a function $\Vt^{({\rm cl})}(S,T)$ of the new variables
\begin{align}\label{AB}
S&\=-\sigma_2(s,t,u)\=\h(s^2+t^2+u^2)\=s^2+st+t^2\,, \notag\\
T&\=\,-\sigma_3(s,t,u)\=-stu\=st(s+t)\,,
\end{align}
with the standard notation $\sigma_i$ for the $i$-th elementary symmetric polynomial, here with $\sigma_1=0$. This new function $\Vt^{({\rm cl})}$ then has poles and zeros along the lines $T=-nS+n^3$ and $T=nS-n^3$,
respectively, where again~$n$ runs over~$\N$.  The Taylor expansion of its logarithm at the origin
is described by the following proposition, in which~$\z(s)$ denotes the Riemann zeta function.
\begin{prop}\label{Proposition 1}
For $S,\,T\in\C$ with $|S|+|T|<1$ one has
\begin{equation}\label{AC}
\Vt^{({\rm cl})}(S,T)\=\exp\bigg(2\sum_{\substack{p,\,q\geq 0\\ \text{q odd}}}
   \binom{p+q-1}{p}\,\z(2p+3q)\,\frac{S^p\,T^q}q\bigg)\;.
\end{equation}
\end{prop}
\begin{proof} From the Weierstrass product expression of $\Gamma$ we can write $V^{({\rm cl})}(s,t,u)$ as an absolutely convergent\footnote{This is true because $s+t+u=0$.} infinite product:
\begin{equation*}
V^{({\rm cl})}(s,t,u) \=\prod_{n=1}^\infty\frac{(n-s)(n-t)(n-u)}{(n+s)(n+t)(n+u)}
\=\prod_{n=1}^\infty \frac{n^3-nS+T}{n^3-nS-T} \;,
\end{equation*}
But for $n\in\N$ and $n|S|+|T|<n^3$ we have
\begin{equation*}
\log\biggl(\frac{n^3-nS+T}{n^3-nS-T}\biggr) 
  \=2\,\sum_{\substack{q\geq 1\\ \text{q odd}}}\frac1q\,\biggl(\frac T{n^3-nS}\biggr)^q 
 \=2\,\sum_{\substack{p,\,q\geq 0\\ \text{q odd}}}\frac1q\,\binom{p+q-1}p\,\frac{S^p\,T^q}{n^{2p+3q}}\;.
\end{equation*}
Equation~(\ref{AC}) follows immediately.\\
\end{proof}
\begin{cor}
The function $V^{({\rm cl})}(s,t,u)=\Vt^{({\rm cl})}(S,T)$ has an expansion
\begin{equation}\label{AD}
\Vt^{({\rm cl})}(S,T) \= 1\;+\; \sum_{p\ge0,\;q\ge1}e_{p,q}\,S^p\,T^q  
\end{equation}
with coefficients $e_{p,q}$ in the ring $\mathcal{R}$ of odd Riemann zeta values.
\end{cor}
More explicitly, by expanding the exponential in~(\AC), we find the formulas
\begin{align*}
   e_{p,1}&\= 2\,\z(2p+3)\,, \\
    e_{p,2}&\= 2\sum_{\substack{p_1,\,p_2\geq 0\\ p_1+p_2=p}}\z(2p_1+3)\,\z(2p_2+3)\,, \\
    e_{p,3}&\= \frac43\sum_{\substack{p_1,\,p_2,\,p_3\geq 0\\ p_1+p_2+p_3=p}}\,\z(2p_1+3)\,\z(2p_2+3)\,\z(2p_3+3)
         \,+\,\frac{(p+1)(p+2)}3\,\z(2p+9)\,,
\end{align*}
for the first few values of~$q$, and in general a formula expressing $e_{p,q}$ as a linear combination of products of
at most $q$ odd Riemann zeta values, always with total weight $2p+3q$.

\subsection{The Taylor coefficients of $V^{({\rm cl})}$ as multiple zeta values}\label{SSec:CoeffVira}

The formulas just given grow rapidly in complexity as~$q$ grows, with $e_{p,q}$ having $\,\O(p^{q-1})\,$ terms.
But by using the ``odd summation formula'' given in~\cite{GKZ}, one can rewrite the formula for $e_{p,2}$ in the form
\begin{equation*}
e_{p,2}\= (2p+3)\,\z(2p+6)\;-\;4\,\z(1,2p+5)\,.
\end{equation*}
This suggests that perhaps also the coefficients $e_{p,q}$ for larger values of~$q$ might be more
simply expressible in terms of \MZVs, and indeed in calculations carried out with Herbert Gangl
we found experimentally (for many~$p$ and numerically to high precision) the formulas
\begin{align*}
e_{p,3}&\=(p+1)(p+3)\,\z(2p+9)\,-\,2\,(2p+5)\,\z(1,2p+8)\,+\,2\,\z(2,2p+7)\,+\,8\,\z(1,1,2p+7)\,,\notag\\
e_{p,4}&\=\frac16\,(p+1)(p+2)(2p+9)\,\z(2p+12)\,-\,2\,(p^2+6p+10)\,\z(1,2p+11)\,+\,(2p+7)\,\z(2,2p+10) \notag\\
  &\,+\,4\,(2p+7)\,\z(1,1,2p+10)\,-\,4\,\z(1,2,2p+9)\,-4\,\z(2,1,2p+9)\,-\,16\,\z(1,1,1,2p+9)
\end{align*}
for $q=3$ and $q=4$.  Looking at these formulas carefully, one observes that the last two terms in the 
formula for $e_{p,3}$ and the last five terms in the formula for $e_{p,4}$ can be written in the form $2Z(2p+7,2)$
and $(2p+7)Z(2p+10,2)-2Z(2p+9,3)$, where $Z(k,r)$ denotes the special sum of \MZVs\ of weight $k+r$ defined by~(\ref{Z}). Now looking for a similar formula for $q=5$, we find
\begin{align*}
 e_{p,5} &\= \frac1{12}(p+1)(p+2)(p+3)(p+6)\,Z(2p+15,0) \,-\, \frac16(2p+11)(p^2+5p+12)\,Z(2p+14,1) \notag \\
      &\quad \; +\, (p^2 + 8p + 19)\,Z(2p+13,2) \,-\, (2p+9)\,Z(2p+12,3) \,+\, 2\,Z(2p+11,4)\,.
\end{align*}
On the basis of these special cases we (independently) conjectured the general formula for the coefficients
$e_{p,q}$ given in the following theorem, whose proof is the main object of this section of the paper.  
\begin{thm}\label{Theorem 1}
For all $p\ge0$ and $q\ge1$, the coefficient $e_{p,q}$ in~(\ref{AD}) is an integral
linear combination of multiple zeta values of weight $2p+3q$ and depth $\le q$ with last argument
larger than $2p+2q$ and all other arguments equal to~$1$ or~$2$.  More precisely, we have
\begin{equation}\label{AE}
e_{p,q}\=\sum_{r=0}^{q-1}(-1)^r\,C(p,q,r)\,Z(2p+3q-r,r)
\end{equation}
where $Z(k,r)$ is defined by~(\ref{Z}) and the coefficients $C(p,q,r)$ by the generating function
\begin{equation}\label{AF}
\sum_{q=r+1}^\infty C(p,q,r)\,y^{q-1}\=\frac{2+y}{(1-y)^{p+1}}\,\biggl(\frac{1-\sqrt{1-4y}}2\biggr)^r
\qquad(r\in\Z_{\ge0},\,\,\,p\in\mathbb{Z})\,.
\end{equation}
\end{thm}
Before proceeding any further, we say a few words about the coefficients $C(p,q,r)$, since the definition~(\ref{AF})
is not very enlightening. For~$r=0$ they are given explicitly by 
\begin{equation}\label{AH}
C(p,q,0) \=\frac{2p+3q-3}p\,\binom{p+q-2}{q-1}\=2\,\binom{p+q-1}{q-1}\,+\,\binom{p+q-2}{q-2}\,,
\end{equation}
while for small values of $q-r$ we have the polynomial formulas
\begin{align}\label{AI}
C(p,q,q-1) &\= 2\,,  \\
 C(p,q,q-2) &\= 2p + 2q - 1\,, \notag\\
 C(p,q,q-3) &\= (p+q)^2\,-\,2p\,-\,6\,. \notag
\end{align}
In general, $C(p,q,q-d-1)$ for a fixed value of  $d\ge0$ is a polynomial of degree $d$ in $p$ and~$q$
with top-degree term equal to $\,2(p+q)^d/d!\,$ and in fact a polynomial of degree $d/2$ in
$(p+q+(d-2)/2)^2$ and $p$. (For $d$ odd this statement means that it is $(p+q+(d-2)/2)$ times
a polynomial of degree $(d-1)/2$ in $(p+q+(d-2)/2)^2$ and $p$.) For instance, 
\begin{equation*}
C(p,q,q-6)=\frac1{60}\biggl(p+q+\frac32\biggr)^5 \,-\, \biggl(\frac13p+\frac{25}{24}\biggr)\biggl(p+q+\frac32\biggr)^3
        \,+\,\biggl(p^2 + \frac{85}{12}p + \frac{2003}{320}\biggr)\biggl(p+q+\frac32\biggr)\,.
\end{equation*}
The coefficients $C(p,q,r)$ satisfy the two Pascal's-triangle-like recursion relations
\begin{align*}
 C(p,q,r)&\= C(p-1,q,r)\,+\,C(p,q-1,r)\,, \notag \\
   C(p,q,r)&\= C(p,q,r+1)\,+\,C(p,q-1,r-1)\,,
\end{align*}
the second of which, together with the initial values~(\ref{AH}) and~(\ref{AI}), determine them completely.
The generating function~(\ref{AF}) also gives the closed formula
\begin{equation*}
C(p,q,r) \= \sum_{n=r}^{q-1}\biggl(\binom{2n-r}{n}-2\binom{2n-r-1}{n}\biggr)\,C(p,q-n,0)
\end{equation*}
with $C(p,q',0)$ given by~(\ref{AH}), making the polynomial nature of $C(p,q,r)$ as a function of~$p$ evident. 

As already mentioned in the introduction, the same expression
as appears in~(\ref{AE}) will turn up in connection with the leading
terms of the modular graph functions studied in Section~3. However, there is a surprise: in the application there we will need
not only the values of $e_{p,q}$ for $p\ge0$, $q\ge1$ appearing in
the Taylor expansion~(\ref{AD}), but also the values in the range $-q<p<0$.
(Notice that the definition~(\ref{AE}) makes sense in this range, since
the first argument $2p+3q-r$ of $Z(\,\cdot\,,r)$ is still~$\ge2$.)
For the application it is important that also these values, like
the ones occurring in Theorem~\ref{Theorem 1}, belong to the odd zeta-value
ring~$\mathcal{R}$, but in fact numerical calculations revealed the
following much stronger statement, whose proof will be given in Section~\ref{SSec:epq=0}:

\begin{thm}\label{Theorem 3}
The numbers $e_{p,q}$ defined by~(\ref{AE}) vanish for $\,-q<p<0\,$.
\end{thm}

\subsection{Proof of Theorem~\ref{Theorem 1}}\label{SSec:ProofVira}

We now turn to the proof of Theorem~\ref{Theorem 1}.  We were unable to find a ``synthetic'' proof which leads to 
the formula~(\ref{AE}) in a natural way beginning from formula~(\ref{AA}) or~(\ref{AC}).  Instead, we shall start from 
the empirically discovered formula~(\ref{AE}) for $e_{p,q}$, define a function $E(s,t,u)=\widetilde{E}(S,T)$ as the 
right-hand side of~(\ref{AD}) with these coefficients, and by studying the properties of this new function 
show that it is equal to the original function~$V^{({\rm cl})}(s,t,u)$. The starting point is the generating series
\begin{equation}\label{AJ}
\sum_{r=0}^\infty Z(k,r)\,t^r\=\sum_{n=1}^\infty\frac1{n^k}\,\prod_{0<m<n}\biggl(1\,+\,\frac{2t}m\,+\,\frac{t^2}{m^2}\biggr)
 \=\sum_{n=1}^\infty\frac1{n^k}\,\binom{n+t-1}{n-1}^2 
\end{equation}
for the numbers defined by~(\ref{Z}).  Together with the generating series
\begin{equation*}
\sum_{p\ge0,\;q\ge r+1} C(p,q,r)\,x^p\,y^{q-1}\= \frac{2+y}{1-x-y}\,\biggl(\frac{1-\sqrt{1-4y}}2\biggr)^r\qquad(r\ge0)\,,
\end{equation*}
which is an immediate consequence of the definition~(\ref{AF}), this gives
\begin{align}\label{AK}
\widetilde{E}(S,T) &\= 1\,+\,\sum_{p\ge0,\;q>r\ge0}(-1)^r\,C(p,q,r)\,Z(2p+3q-r,r)\,S^p\,T^q \notag\\
  &\= 1\,+\,\sum_{n=1}^\infty\sum_{p\ge0}\sum_{q>r\ge0}(-1)^r\,C(p,q,r)\,
      \frac{S^p\,T^q}{n^{2p+3q-r}}\,\text{Coeff}_{t^r}\biggl[\binom{n+t-1}{n-1}^2\biggr] \notag\\
&\= 1\,+\,T\,\sum_{n=1}^\infty\sum_{r=0}^\infty(-n)^r\,\frac{2+T/n^3}{n^3-nS-T}\,\biggl(\frac{1-\sqrt{1-4T/n^3}}2\biggr)^r\,
      \text{Coeff}_{t^r}\biggl[\binom{n+t-1}{n-1}^2\biggr]\notag\\
&\= 1\,+\,T\,\sum_{n=1}^\infty\frac{2+T/n^3}{n^3-nS-T}\, \binom{\h\bigl(n+\sqrt{n^2-4T/n}\,\bigr)-1}{n-1}^2\;. 
\end{align}
Note that this can be written more uniformly, although even less beautifully, as
\begin{equation*}
\widetilde{E}(S,T)\=\sum_{n=0}^\infty \frac{2n^3+T}{2T}\;\frac{n^3-2T+\sqrt{n^3(n^3-4T)}}{n^3-nS-T}\;
\binom{\h\bigl(n+\sqrt{n^2-4T/n}\,\bigr)}n^2\;.
\end{equation*}
Alternatively, and more attractively, we can use the simple identity
\begin{equation}\label{AL}
\binom{a-1}{n-1}^2\=\prod_{0<j<n}\biggl(1\,-\,\frac aj\biggr)\biggl(1\,-\,\frac a{n-j}\biggr)
   \=\prod_{0<j<n}\biggl(1\,-\,\frac{a(n-a)}{j(n-j)}\biggr)
\end{equation}
to rewrite~(\ref{AK}) as
\begin{equation}\label{AM}
\widetilde{E}(S,T)\= 1\,+\,T\,\sum_{n=1}^\infty\frac{2+T/n^3}{n^3-nS-T}\,\prod_{j=1}^{n-1}\biggl(1\,-\,\frac T{nj(n-j)}\biggr)\;. 
\end{equation}
Going back to the original variables~$s$,~$t$ and~$u$, we find that we have reduced the proof of the theorem 
to the proof of the following proposition, which is of some interest in itself.
\begin{prop}\label{Proposition 2} 
For $(s,t,u)\in\X$ we have the convergent series representation
\begin{equation}\label{AN}
V^{({\rm cl})}(s,t,u)\= 1\,-\,stu\,\sum_{n=1}^\infty\frac{2-stu/n^3}{(n+s)(n+t)(n+u)}
 \,\prod_{j=1}^{n-1}\biggl(1\,+\,\frac{stu}{nj(n-j)}\biggr).
\end{equation}
\end{prop}
\begin{proof} We continue to denote the function on the right by $E(s,t,u)$.  We first observe that the series in~(\ref{AN})
converges like $\sum n^{-3}$ (the product over~$j$ is bounded, and in fact is $\,1+\O\bigl(n^{-2}\log n\bigr)$, as $n\to\infty$),
so defines a meromorphic function in~$\X$ with simple poles along the lines $s=-n$, $t=-n$, $u=-n$ ($n\in\N$) and no other poles. We
compare the residues of~$V^{({\rm cl})}$ and~$E$ at $s=-n\in\mathbb{Z}_{<0}$ (and with~$t,\,u$ generic---i.e.~both not belonging to $\mathbb{Z}_{<0}$---with sum~$n$). For~$E$ we have
\begin{align*}
\text{Res}_{s=-n} \big[E(s,t,u)\,ds\big] &\=\,ntu\,\frac{2+tu/n^2}{(n+t)(n+u)}\,\prod_{j-1}^{n-1}\biggl(1\,-\,\frac{tu}{j(n-j)}\biggr) \notag\\
  &\= \frac{tu}n\,\binom{t-1}{n-1}^2\= (-1)^{n-1}\,n\,\binom{t}n\binom{u}n\,,
\end{align*}
where we have used the identities $(n+t)(n+u)=2n^2+tu$ and~(\ref{AL}).  For~$V^{({\rm cl})}$ the computation is even simpler: since the residue
of $\G(1+s)$ at its simple pole at $s=-n$ is $\frac{(-1)^{n-1}}{(n-1)!}$ we have
\begin{equation*}
\text{Res}_{s=-n} \big[V^{({\rm cl})}(s,t,u)\,ds\big]\= \frac{(-1)^{n-1}}{(n-1)!}\,\frac{\G(1+t)\,\G(1+u)}{\G(1+n)\,\G(1+t-n)\,\G(1+u-n)}\
       \= (-1)^{n-1}\,n\,\binom{t}n\binom{u}n\,.
\end{equation*}
By symmetry the residues of the functions~$V^{({\rm cl})}$ and~$E$ at the other poles $t=-n$ and $u=-n$ also agree, so the difference 
of~$V^{({\rm cl})}$ and~$E$ is holomorphic in~$\X$, and the difference of~$\Vt^{({\rm cl})}$ and~$\widetilde{E}$ is holomorphic in~$\C^2$. To complete 
the proof, we consider the growth of $\Vt^{({\rm cl})}$ and $\widetilde{E}$ for $|S|\to\infty$ with~$T$ fixed.  Rewriting~(\ref{AM}) as
\begin{equation*}
\widetilde{E}(S,T)\=1+\sum_{n\ge1}\frac{r_n(T)}{S-n^2+T/n}\,,\qquad 
   r_n(T)=-\frac{T(2+T/n^3)}n\prod_{j=1}^{n-1}\bigl(1-\frac T{nj(n-j)}\bigr)
\end{equation*}
and noting that $r_n(T)=\O(1/n)$ as $n\to\infty$ with~$T$ fixed, we find
\begin{equation*}
\widetilde{E}(S,T)\= \begin{cases} 1\,+\,\O\biggl(\dfrac{\log S}S\biggr) &\quad\text{for }S\in\C\smallsetminus\bigcup_{n\ge1}D_n,\quad|S|\to\infty, \\
1\,+\,\dfrac{r_n(T)}{S-n^2+T/n}\,+\,\O\biggl(\dfrac{\log S}S\biggr) &\quad\text{for }S\in D_n,\quad n\to\infty, \end{cases}
\end{equation*}
where~$D_n$ denotes the disk with center~$n^2$ and radius~$n$.  For $V^{({\rm cl})}$, we note that the only way that~$|S|$ can go to 
infinity with $T\ne0$ fixed is for exactly one of the three variables~$s$,~$t$ and~$u$, say~$s$, to go to 0 like $\,\O(1/S)$
and the other two to go to infinity like $\pm\sqrt S$.  Then $\G(1+s)/\G(1-s)$ equals $1+\O(1/S)$ and each of
$\G(1+t)/\G(1-u)$ and $\G(1+u)/\G(1-t)$ equals $1+\O((\log S)/S)$,  except for an additional pole term $r_n(T)/(S-n^2+T/n)$ in the
vicinity of the pole $S=n^2-T/n$ (corresponding to $t=-n$ or $u=-n$), so we get exactly the same estimate for $\Vt^{({\rm cl})}(S,T)$ as the
one just given for~$\widetilde{E}(S,T)$.  It follows that the difference between~$\Vt^{({\rm cl})}$ and~$\widetilde{E}$ for~$T$ fixed is a holomorphic
function of~$S$ which is~$\,\text o(1)\,$ as $|S|\to\infty$, and hence vanishes identically. \\
\end{proof}

\subsection{Complex beta function and a related symmetry property}\label{SSec:Beta}

As explained at the beginning of the section, the Virasoro function~$V^{({\rm cl})}$ is essentially given by the complex beta  function~$\beta_\mathbb{C}$ defined by the integral~(\ref{ComplexBeta}). Here we give an alternative expression for~$\beta_\mathbb{C}$ as a sum of two hypergeometric functions~${}_3F_2$, which will be used in the proofs of Theorem~\ref{Theorem 3} and later also of Proposition~\ref{ThmMainClosed}.
\begin{prop}\label{Prop2}
For $(s,t,u)\in\X$ such that $\emph{Re}(s)\geq 0$, $\emph{Re}(t)\geq 0$ and $\emph{Re}(u)\geq -1$ we have
\begin{equation}\label{sumrepbeta}
\beta_{\mathbb{C}}(s,t)\=\sum_{n\geq 0}\binom{u}{n}^2\bigg(\frac{1}{n+s}+\frac{1}{n+t}\bigg).
\end{equation}
\end{prop}
\begin{proof} 
By equation~(\ref{betaShapiro}) and the properties of~$V^{({\rm cl})}$ we can write the integral defining $\beta_{\mathbb{C}}$ as
\begin{equation*}
\beta_{\mathbb{C}}(s,t)\=-\frac{1}{2\pi i}\,\int_{\mathbb{C}}\,|z|^{2s-2}\,|1-z|^{2u}\,dz\,d\overline{z}\=-\frac{1}{2\pi i}\,\bigg(\int_{|z|\leq 1}\,+\,\int_{|z|\geq 1}\bigg)\,|z|^{2s-2}\,|1-z|^{2u}\,dz\,d\overline{z}.
\end{equation*}
We define
\begin{equation*}
\Psi(x,y)\;:=\;-\frac{1}{2\pi i}\,\int_{|z|\leq 1}\,|z|^{2y-2}\,|1-z|^{2x}\,dz\,d\overline{z},
\end{equation*}
and we use the change of variable $z\rightarrow 1/z$ to write
\begin{equation*}
\beta_{\mathbb{C}}(s,t)\=\Psi(u,s)\,+\,\Psi(u,t).
\end{equation*}
Using polar coordinates and the binomial theorem we obtain
\begin{align*}
\Psi(x,y)&\=\frac{1}{\pi}\,\int_0^1\,r^{2y-1}\,dr\,\int_0^{2\pi}\,|1-re^{i\theta}|^{2x}\,d\theta \notag \\
&\=\frac{1}{\pi}\,\sum_{m,n\geq 0}\,(-1)^{m+n}\,\binom{x}{n}\binom{x}{m}\,\int_0^1\,r^{2y+m+n-1}\,dr\,\int_0^{2\pi}\,e^{i(m-n)\theta}\,d\theta \notag \\
&\=\sum_{n\geq 0}\,\binom{x}{n}^2\frac{1}{n+y}.
\end{align*}
\end{proof}
\begin{cor}
The real numbers $f_{\mu,\nu}$ $(\mu,\,\nu\ge0)$ defined by
\begin{equation}\label{A}
f_{\mu,\nu} \= \sum_{p+r=\mu}(-1)^r\,\bigg(\binom{\nu+p}{\nu}\;+\;(-1)^\nu\delta_{p,0}\bigg)\,Z(\nu+p+3,r)
\end{equation}
are symmetric in $\mu$ and~$\nu$ and belong to the ring~$\mathcal{R}=\mathbb{Q}[\zeta(3),\zeta(5),\zeta(7),\ldots ]$.
\end{cor}
\begin{proof}
The generating function of the numbers $f_{\mu,\nu}$ is given by
\begin{align}\label{Aa}
 \sum_{\mu,\,\nu\ge0}f_{\mu,\nu}\,t^\nu u^\mu  
 &\=  \sum_{h=0}^\infty\big((t+u)^h\;+\;(-t)^h\big)\,\sum_{r=0}^\infty Z(h+3,r)\,(-u)^r \qquad(h=\nu+p) \notag\\
&\= \sum_{n=1}^\infty\frac1{n^2}\,\Bigl(\frac1{n+s}\;+\;\frac1{n+t}\Bigr)\,\binom{n-u-1}{n-1}^2
\= \frac{1\,-\,V^{\text(cl)}(s,t,u)}{stu}\,,
\end{align}
where in the second line we have set $s=-t-u$ and used equation~(\ref{AJ}) for the first equality and 
equations~(\ref{sumrepbeta}) and~(\ref{betaShapiro}) for the second. This generating function is therefore symmetric not only in~$t$ and~$u$ 
as required, but in all three variables~$s$, $t$,~and~$u$, and every~$f_{\mu,\nu}$ can be expressed in terms of the Taylor coefficients of~$V^{\text(cl)}(s,t,u)$, which belong to $\mathcal{R}$.\\
\end{proof}

\subsection{Structure of the spaces $\mathcal{D}_w$ and $\mathcal{E}_w$ and proof of Theorem~\ref{Theorem 3}}\label{SSec:epq=0}

The ultimate goal of this section is to prove Theorem~\ref{Theorem 3}, but we first focus on proving the weaker statement $e_{p,q}\in \mathcal{R}$, needed for our main application in the next section. We will also stress how these results shed more light on the structure of the rational vector spaces\footnote{The range for the indices $p,q$ in this definition of $\mathcal{E}_w$ includes also the cases with $-q<p<0$, as opposed to the definition~(\ref{calEcalZ1}) given in the introduction where we set $p\geq 0$, but we use the same notation because Theorem~\ref{Theorem 3} says that the two spaces coincide.} of weight-$w$ \MZVs\
\begin{equation}\label{calEcalZ2}
\mathcal{E}_w\,:=\,\big{\langle} \,e_{p,q}\;\big{|}\;q>0,\,p+q>0,\, 2p+3q=w\,\big{\rangle}_{\mathbb{Q}} \,\,\, \subseteq \,\,\, \mathcal{D}_w\,:=\,\big{\langle} \,Z(w-r,r)\;\big{|} \;0\leq r\leq w-2\,\big{\rangle}_{\mathbb{Q}}.
\end{equation}

The main ingredient of the proof is to consider for every integer $w\geq 3$ also the vector space
\begin{equation*}
\mathcal{F}_w\;:=\;\big{\langle} \,f_{\mu,\nu}\;\big{|}\;\mu,\nu\geq 0,\,\mu+\nu+3=w\,\big{\rangle}_{\mathbb{Q}} \,\,\, \subseteq \,\,\, \mathcal{E}_w\,\cap\, \mathcal{R} 
\end{equation*}
spanned by the numbers $f_{\mu,\nu}$ defined by~(\ref{A}) and contained in $\mathcal{E}_w\,\cap\, \mathcal{R}$ by~(\ref{Aa}). We will show that it coincides with $\mathcal{E}_w$.
\begin{prop}\label{PropPoint1}
For all~$w\geq 3$ and all $0\leq \nu\leq w-3$ we have
\begin{equation}\label{etofnew}
f_{w-3-\nu,\nu}\=\sum_{0<j<w/2}\,\lambda_{j,\nu}(w)\,e_{3j-w,w-2j},
\end{equation}
where $\lambda_{j,\nu}(w)$ are the polynomials defined by the generating series
\begin{equation}\label{GenSerLam}
\sum_{\nu\ge 0,j\geq 1}\,\lambda_{j,\nu}(w)\,X^{j-1}\,Y^{\nu}\;:=\;\frac{(1-t)^{w-2}}{(1-3t)\,\big(1-t\,(1+Y+Y^2)\big)} \qquad \,\,\, (X\=t\,(1-t)^2).
\end{equation}
\end{prop}

\begin{proof}
Comparing for every~$r$ the coefficients of $Z(w-r,r)$ and substituting $n=w-3-r$, the statement is reduced to proving for all $0\leq n\leq w-3$ and $0\leq \nu\leq w-3$ the identity
\begin{equation*}
\Big(1\,+\,(-1)^\nu\,\delta_{\nu,n}\Big)\,\binom{n}{\nu}\=\sum_{0<j<w/2}\,\lambda_{j,\nu}(w)\,C(3j-w,w-2j,w-3-n).
\end{equation*}
We consider the generating function $\Lambda_j(w,Y):=\sum_{\nu\geq 0}\lambda_{j,\nu}(w)Y^\nu$, and we prove the stronger statement that for all $n\in\mathbb{Z}_{\geq 0}$ and for all\footnote{Notice that $C(p,q,r)$ is well-defined by~(\ref{AF}) for any $(p,q,r)\in\mathbb{C}^3$ such that $q-r\in\mathbb{Z}$, and vanishes if $q-r\leq 0$.} $w\in\mathbb{C}$
\begin{equation}\label{B}
\sum_{1\le j\le\frac{n+2}2} \Lambda_j(w,Y)\,C(3j-w,w-2j,w-3-n) \= (1+Y)^n\;+\;(-Y)^n.
\end{equation}
Substituting $t=\frac{u^2}{1-u+u^2}$  and $X=\frac{u^2(1-u)^2}{(1-u+u^2)^3}$ in the expansion~(\ref{GenSerLam}) gives 
\begin{equation}\label{D}
\sum_{j\ge1} \Lambda_j(w,Y)\,\frac{(1+u)(2-u)(1-2u)\,u^{2j-2} }{(1-u)^{w-2j}(1-u+u^2)^{3j-w+1}}
  \= \frac{1}{1+uY} \;+\; \frac{1}{1-u(1+Y)}\,.
\end{equation}
On the other hand, substituting $y=u(1-u)$ in the definition of $C(p,q,r)$ gives the formula
\begin{align*}
 C(p,q,r) &\= \text{Res}_{y=0}\bigg[\frac{2+y}{(1-y)^{p+1}}\,\biggl(\frac{1-\sqrt{1-4y}}2\biggr)^r\,\frac{dy}{y^{q}}\bigg] \\
  & \= \text{Res}_{u=0}\bigg[\frac{(1+u)(2-u)(1-2u)}{(1-u)^q(1-u+u^2)^{p+1}}\,\frac{du}{u^{q-r}}\bigg]\,.
\end{align*}
Equation~(\ref{B}) then follows by comparing the coefficients of~$u^n$ on both sides of~(\ref{D}).\\
\end{proof}
At first sight, this proposition only implies the already known inclusion $\mathcal{F}_w\subseteq\mathcal{E}_w$. We show that it implies also the opposite inclusion. Rewriting $\Lambda_j(w,Y)$ from the previous proof as a residue first at $X=0$ and then at~$t=0$ gives
\begin{equation}\label{E}
\Lambda_j(w,Y) \= \text{Res}_{t=0}\bigg[\frac{(1-t)^{w-2j-1}}{1-t(1+Y+Y^2)}\,\frac{dt}{t^j}\bigg] \= \sum_{r=0}^{j-1}(-1)^r\binom{w-2j-1}r\,(1+Y+Y^2)^{j-1-r} \,. 
\end{equation}
From this formula we find that for each $j\geq 1$ the top coefficients $\lambda_{j,\nu}(w)$ read
\vspace{5mm}
\bgroup
\def\arraystretch{1.5}
\begin{center}
\begin{tabular}{  c | c | c | c | c }
  $\nu$ & $> 2j-2$ & $2j-2$ & $2j-3$ & $2j-4$ \\ [0.5ex]
  \hline
  $\lambda_{j,\nu}(w)$  & 0 & 1 & $j-1$ & $\binom{j+2}{2}-w$ \\
\end{tabular}\,\, ,
\end{center}
\egroup
\vspace{5mm}
\noindent as one can see from the example
\[
\Big(\lambda_{j,k-1}(12)\Big)_{\substack{j=1,\ldots ,5\\k=1,\ldots ,10}}\=
\begin{pmatrix}
  1 & 0 & 0 & 0 & 0 & 0 & 0 & 0 & 0 & 0 \\
  -6 & 1 & 1 & 0 & 0 & 0 & 0 & 0 & 0 & 0 \\
  6 & -3 & -2 & 2 & 1 & 0 & 0 & 0 & 0 & 0 \\
  0 & 0 & 0 & 1 & 3 & 3 & 1 & 0 & 0 & 0 \\
  0 & 1 & 4 & 9 & 13 & 13 & 9 & 4 & 1 & 0 \\
 \end{pmatrix}\,.
\]
The triangular shape of the matrix $\big(\lambda_{j,k-1}(w)\big)_{j,k}$ shows that it has full rank, and therefore proves the desired opposite inclusion:
\begin{cor}
For every integer $w\geq 3$ we have $\mathcal{E}_w=\mathcal{F}_w$ and therefore $\mathcal{E}_w\subset\mathcal{R}$.
\end{cor}
Notice that, assuming standard conjectures on the structure of the $\mathbb{Q}$-algebra of \MZVs, the numbers $Z(k,r)$ do not in general belong to~$\mathcal{R}$ (e.g. $Z(2k,0)=\zeta(2k)$), so this result implies that in general the inclusion~(\ref{calEcalZ2}) is strict\footnote{Numerical experiments suggest that the inclusion is strict for all $w\neq 3$.}.

As mentioned previously, this corollary is already sufficient for the application discussed in the next section. To obtain the full statement of Theorem~\ref{Theorem 3} we need to study the polynomials~$\lambda_{j,\nu}(w)$ in more detail.
We first remark that the striking symmetry of the last two rows of the matrix above is not an accident, and it generalizes as follows:
\begin{lemma}\label{PropPoint2}
For $w\in\{2j+1,\ldots,3j\}$ and $0\le\nu\le w-3$ we have 
$\lambda_{j,\nu}(w)=\lambda_{j,w-3-\nu}(w)$.
\end{lemma}
\begin{proof}
If $w$ is an integer in the range stated, then the sum~(\ref{E}) terminates at $r=w-2j-1$ and we find
\begin{equation*}
\Lambda_j(w,Y) \= (Y+Y^2)^{w-2j-1}\,(1+Y+Y^2)^{3j-w}\,,
\end{equation*}
from which the asserted symmetry is obvious.\\
\end{proof}


\begin{proof}[Proof of Theorem~\ref{Theorem 3}]
Equation~(\ref{etofnew}) together with the symmetry statements in Lemma~\ref{PropPoint2} and~in the corollary of Proposition~\ref{Prop2} gives the relations $\sum_{0<j<w/3}\lambda_{j,\nu}^-(w)\,e_{3j-w,w-2j}=0$ for 
all~$w\in\mathbb{Z}_{\ge3}$, where~$\lambda_{j,\nu}^-(w)$ denotes the antisymmetrized coefficient  
$\lambda_{j,\nu}(w)-\lambda_{j,w-3-\nu}(w)$.  To show the vanishing of the $e_{p,q}$ with $-q<p<0$
it therefore suffices to show that the matrix $\bigl(\lambda_{j,\nu}^-(w)\bigr)_{0<j<w/3,0\le\nu\le w-3}$
has maximal rank $\big\lfloor\frac{w-1}3\big\rfloor$, or equivalently that the antisymmetrized
polynomials 
\begin{equation*}
\Lambda_{w,j}^-(Y)\;:=\;\Lambda_j(w,Y)\,-\,Y^{w-3}\,\Lambda_j(w,1/Y) \qquad (0<j<w/3)
\end{equation*}
are linearly
independent.  If we identify the space $V_{w-3}=\langle Y^\nu\rangle_{0\le\nu\le w-3}$
with the weight~$w-3$ part of the vector space $V=\Q[s,t,u]/(s+t+u=0)$ via the homogenization
operator $Y^\nu\mapsto t^\nu u^{w-3-\nu}$, then equation~(\ref{E}) implies that the space $L_w:=\langle\Lambda_j(w,Y)\rangle_{0<j<w/2}$ coincides with the subspace of~$V_{w-3}$ symmetric under the interchange of~$s$~and~$t$, and hence that the space $L_w^-$
spanned by the~$\Lambda_{w,j}^-$ with $0<j<w/2$ (or equivalently, in view of Lemma~\ref{PropPoint2}, by
the $\Lambda_{w,j}^-$ with $0<j<w/3$) is isomorpic to the image of~$L_w$ under $1-\iota$,
where~$\iota$ denotes the involution~$t\leftrightarrow u$.  But the kernel of the map 
$1-\iota:L_w\twoheadrightarrow L_w^-$ is just the space~$V_{w-3}^{\mathfrak{S}_3}$ of homogeneous polynomials 
of degree~$w-3$ that are symmetric in all three variables $s,\,t,\,u$.  Hence the dimension 
of~$L_w^-$ equals the coefficient of~$x^{w-3}$ in the generating series 
$\frac1{(1-x)(1-x^2)}\,-\,\frac1{(1-x^2)(1-x^3)}$, which is~$\big\lfloor\frac{w-1}3\big\rfloor$ as desired.\\
\end{proof}

The statement of Theorem~\ref{Theorem 3} gives relations between the numbers $Z(k,r)$, and it therefore reduces the upper bound on the dimension of the space $\mathcal{D}_w$ from $w-1$ to $\big\lfloor\tfrac{2w}{3}\big\rfloor$. Experimentally, this bound seems to be sharp for $w\geq 10$, i.e. as soon as $\big\lfloor\tfrac{2w}{3}\big\rfloor$ is lower than $\dim(\mathcal{Z}_w)$, while for $w\leq 9$ $\dim(\mathcal{D}_w)=\dim(\mathcal{Z}_w)$. As for the dimension of $\mathcal{E}_w$, Theorem~\ref{Theorem 3} gives the upper bound $\big\lfloor\tfrac{w+1}{6}\big\rfloor$ for $w\neq 3\,\,\text{mod}\,\,6$ and $\big\lfloor \tfrac{w}{6}\big\rfloor+1$ for $w=3\,\,\text{mod}\,\,6$. Experimentally, this bound seems to be sharp for all $w\geq 2$. Note that this dimension coincides with that of the space of cusp forms of weight~$2w+6$ on $\text{SL}_2(\mathbb{Z})$, but we do not know whether there is a direct connection between the space $\mathcal{E}_w$ and modular forms or their period polynomials.

\section{Two-point modular graph functions}\label{Sec:MGF}

\noindent At genus one, the study of closed superstring amplitudes led to a new class of real analytic modular functions associated to graphs, now known as \emph{modular graph functions}~\cite{DGGV, ZerbiniThesis}. The four-point amplitude involves\footnote{More precisely, it is given by integrals of the modular graph functions discussed here over the moduli space $\overline{\mathfrak{M}_{1,1}}$.} contributions from two-vertex, three-vertex and four-vertex graphs. Here we only consider two-vertex modular graph functions, which constitute a family parametrized by an integer $\l\ge0$ (the number of edges between the two vertices). They were first systematically studied in ~\cite{GRV} by Green, Russo and Vanhove and are defined as the average
\begin{equation}\label{BE}
D_\l(\t) \= \text{Av}_z\bigl(G(z,\t)^\l\bigr) \,:=\, \iint_{\C/\Lambda_\t} G(z,\t)^\l\,d\mu(z)\,.
\end{equation}
Here $\t$ is a variable in the complex upper half-plane, $\Lambda_\t\subset\C$ the
lattice $\Z\t+\Z$, $d\mu(z)=\tfrac{i\,dz\,d\overline{z}}{2\,\textup{Im}(\tau)}$ the normalized translation-invariant measure on the elliptic
curve $\C/\Lambda_\t$, and $G(z,\t)$ the Green's function on the torus (see the appendix), which is defined in terms of the odd Jacobi function $\theta_1(z,\tau)$ and the Dedekind function $\eta(\tau)$ by
\begin{equation}\label{DefGreenfct}
G(z,\tau)\=-\log\bigg|\frac{\theta_1(z,\tau)}{\eta(\tau)}\bigg|^2\;+\;\frac{2\pi \,\mbox{Im}(z)^2}{\mbox{Im}(\tau)}.
\end{equation}
The function~$G$ is invariant under translations $z\mapsto z+\o$ ($\o\in\Lambda_\t$) and modular
transformations $\,(z,\t)\mapsto\Bigl(\dfrac z{c\t+d},\,\dfrac{a\t+b}{c\t+d}\Bigr)\,$ with 
$\,\Bigl(\begin{matrix} a&b\\c&d\end{matrix}\Bigr)\in\text{SL}_2(\Z)$, so that the functions $ D_\l(\t)$ are
well-defined and $\,\text{SL}_2(\Z)$-invariant. It was remarked in~\cite{GV2000} that $D_1(\tau)=0$ and $D_2(\tau)=E(2,\tau)$, where $E(s,\tau)$ is the non-holomorphic Eisenstein series. Moreover, the first author proved in some unpublished notes that $D_3(\tau)=E(3,\tau)+\zeta(3)$ (see~\cite{DGV2015A} for a more recent published proof). In general, however, these functions are known to constitute an interesting new family of modular function which goes beyond special values of non-holomorphic Eisenstein series and is conjecturally contained into a class of real analytic modular forms constructed by Brown from iterated Eichler integrals of Eisenstein series~\cite{BrownNewClass}. The behavior at infinity can be found by using the product
expansions of $\theta_1$ and~$\eta$ to write $G(z,\t)$ as an infinite sum of logarithms of the form
$\,\log(1-t)\,$ and then expanding each of these as a series in~$t$ and carrying out the remaining summation
and integration. The result of this computation is the following:
\begin{prop}[Green, Russo, Vanhove, Zagier~\cite{GRV}]\label{TeoGRV}
If we set $Y:=2\pi\,\rm{Im}(\tau)\rightarrow +\infty$,\footnote{Notice that we deviate from the usual notation $y=\pi\,\mbox{Im}(\tau)$ employed in most references, so our formulas look slightly different.} then
\begin{equation*}
D_\l(\tau)\=d_{\l}(Y)\;+\;O(e^{-Y})
\end{equation*}
with 
\begin{equation}\label{Formula2pointclosed}
d_{\l}(Y)\=\bigg(\frac{Y}{6}\bigg)^{\l} {}_2F_1\biggl(1,-\l;\,\frac32;\,\frac32\biggr)\;+\sum_{\substack{a+b+c+m=\l\\m\geq 2}}\frac{\l!\,(2a+b)!}{a!\,b!\,c!}\,\frac{(-1)^{b}}{2^{2a+b}\,6^c}\,Z(2a+b+3,m-2)\,Y^{c-a-1},
\end{equation}
where $Z(k,r)$ are the numbers defined by~(\ref{Z}).
\end{prop}
This proposition implies in particular that all coefficients of $d_{\l}(Y)$ belong to the $\mathbb{Q}$-algebra $\mathcal{Z}$ of \MZVs. The main result of this section is the much stronger assertion given by Theorem~A from the introduction, whose truth was suggested by the numerical calculations\footnote{More precisely, the formulas given in~\cite{GRV} show that all coefficients of $d_\l$ for $0\le\l\le5$ are polynomials in odd zeta values. For $\l=6$ the expression given in~\cite{GRV} involves both even and odd zeta values, but it turned out that this was due 
to an error of transcription in the data in their formula~(B.11) and that in fact also $d_6(Y)$ involved only odd 
zetas, making it reasonable to conjecture the truth of Theorem~\ref{Theorem 2}.} in~\cite{GRV} and whose statement we repeat here:
\begin{thm}\label{Theorem 2}
For every integer $\l\ge0$ the Laurent polynomial $d_\l(Y)$  has coefficients belonging
to the ring $\mathcal{R}$ generated over~$\Q$ by the odd Riemann zeta values.
\end{thm}

Since the numbers $Z(k,r)$ do not in general belong to $\mathcal{R}$, we must study in detail the rational linear combinations of multiple zeta values 
occurring in~(\ref{Formula2pointclosed}).  We first rewrite~(\ref{Formula2pointclosed}). The hypergeometric function appearing there can be written as
\begin{equation*}
_2F_1\biggl(1,-\l;\,\frac32;\,\frac32\biggr) 
  \= \sum_{n=0}^\l \frac{\l!}{(\l-n)!}\,\frac{(-3/2)^n}{\frac32\cdot\frac52\cdots(n+\frac12)}
 \= \l!\,\sum_{c+n=\l} \frac{(-3)^n}{c!\,(2n+1)!!}\,,
\end{equation*}
where $(2n+1)!!$ denotes the double factorial $1\times3\times\cdots\times(2n+1)$. 
Inserting this expression into~(\ref{Formula2pointclosed}), and changing the names of $a$ and $m$ to $k-1$ and $r+2$,
we find that the generating series $\sum_\l d_\l(Y)s^\l/\l!$ factors as the product of a pure
exponential (corresponding to the summation over the variable~$c$) and a power series whose coefficients are considerably
simpler Laurent polynomials in~$Y$, as given in the following proposition.\footnote{It should be possible to deduce this proposition by a direct computation of the leading contribution to the generating series of the integrals~(\ref{BE}) and therefore relate our proof of Theorem~\ref{Theorem 2} to that given in \cite{DG19}.}
\begin{prop}\label{Proposition 4}
The Laurent polynomials $d_\l(Y)$ are given by the generating function
\begin{equation}\label{BJ}
\sum_{\l=0}^\infty  d_\l(Y)\,\frac{s^\l}{\l!} 
   \= e^{sY/6}\, \sum_{n=0}^\infty \,\biggl(\frac{Y^n}{(2n+1)!!} 
  \;+\,\sum_{k=1}^{n-1}\frac{(-1)^{k-1}\,(2k-3)!!\,\g_{n,k}}{Y^k}\biggr)\Bigl(\frac{-s}2\Bigr)^n \,, 
\end{equation}
where $(-1)!!:=1$ and the coefficients $\g_{n,k}$ are defined by
\begin{equation}\label{BK}
\g_{n,k} \= \sum_{r=0}^{n-k-1}(-2)^{r+2}\,\binom{n-r+k-3}{2k-2}\,Z(k+n-r,r)\qquad(0<k<n)\,.
\end{equation}
\end{prop}
Proving Theorem~\ref{Theorem 2} thus reduces to showing that all the numbers $\g_{n,k}$ belong to $\mathcal{R}$. We prove this in two different ways, both based on the stronger fact that the numbers $\g_{n,k}$ can be written in terms of the coefficients $e_{p,q}$ of the Virasoro function $V^{({\rm cl})}$. 

\subsection{Expressing $\g_{n,k}$ in terms of the Taylor
coefficients $e_{p,q}$ of $V^{({\rm cl})}$}\label{SSec:1stProof}

We first check the truth of the assertion for small values of $n-k$.
Comparing equation~(\ref{BK}) with the formulas for $e_{p,q}$ for small~$q$ given in Section~\ref{SSec:CoeffVira}, we find the identities
\begin{align}\label{CA}
\g_{k+1,k} & \= \phantom{3}2\,e_{k-1,1}\;, \notag \\
 \g_{k+2,k} & \=  \phantom{3}4\, e_{k-2,2}\;, \notag \\
 \g_{k+3,k} & \=  \phantom{3}8\,e_{k-3,3}\,+\,2\,(3k)\,e_{k,1}\;, \notag \\
 \g_{k+4,k} & \=  16\,e_{k-4,4}\,+\,4\,(3k-4)\,e_{k-1,2}\;, \notag \\
 \g_{k+5,k} & \=  32\,e_{k-5,5}\,+\,8\,(3k-8)\,e_{k-2,3}\,+\,9k(k+1)\,e_{k+1,1}\;,
\end{align}
proving our claim in these cases. Note that, although each of these formulas is easy to verify, the fact that they exist
is by no means automatic, since, for example, $\g_{k+4,k}$ is a linear combination of \emph{four} values $Z(2k+4-r,r)$ ($0\le r\le3$)
and we only have \emph{two} relevant $e$-values $e_{k-4,4}$ and $e_{k-1,2}$, so that there is no a priori
reason that any linear combination of them should give~$\g_{k+4,4}$. 

Equations~(\ref{CA}) and their successors suggest that the numbers $\g_{n,k}$ are given by
\begin{equation}\label{BO}
\g_{h+k,k} \= \sum_{0\le 2s<h} 2^{h-2s}\,P_s(h,k)\,e_{k-h+3s,h-2s}\qquad(h,\,k>0)
\end{equation}
for some integer-valued polynomials $P_s(k,h)$ of degree~$s$ in~$h$ and~$k$, the first three being
\begin{equation*}
P_0(k,h)=1\,,\;\quad P_1(k,h) = 3k-4h+12\,, \;\quad P_2(k,h) = \frac{(3k-4h+22)^2 +3k+4}2\,.
\end{equation*}
These polynomials grow rapidly in complexity and are not easy to recognize.  Since the coefficients of the
numbers $Z(k,r)$ in~(\ref{BK}) are simpler than those occurring in~(\ref{AE}), we look instead at the inversion of~(\ref{BO}),
which has the form
\begin{equation}\label{BN}
2^{q}\,e_{p,q} \= \sum_{0\le 2s<q} Q_s(p,q)\,\g_{p+2q-s,\,p+q+s}\;
\end{equation}
for certain integer-valued polynomials $Q_s(p,q)$, the first three of which are
\begin{equation*}
Q_0(p,q)=1\,,\;\quad Q_1(p,q) = q-3p-12\,,\;\quad Q_2(p,q)= \frac{(q-3p-3)(q-3p-24)}2 +2q\,.
\end{equation*}
After a little trial and error we recognize these as the coefficients of the generating function
\begin{equation}\label{BM}
\sum_{s=0}^\infty Q_s(p,q)\,t^s \= \frac{1-9t}{(1+3t)^{p+1}(1-t)^q}\;, 
\end{equation}
and by a simple residue calculation we can invert this to find the generating series formula
\begin{equation}\label{BP}
\sum_{s=0}^\infty P_s(h,k)\,\Bigl(\frac x{(1+4x)^3}\Bigr)^s \= \frac1{(1-8x)(1+x)^k(1+4x)^{h-k-1}} 
\end{equation}
for the polynomials $P_s$ as well. 
\begin{prop}\label{Proposition 5}
The numbers $\g_{n,k}$ defined by~(\ref{BK}) for $0<k<n$ and the numbers $e_{p,q}$ defined 
by~(\ref{AE}) for $q>0$ and $p+q>0$ are related by the formulas~(\ref{BO}) and~(\ref{BN}), where the polynomials $P_s(k,h)$ 
and~$Q_s(p,q)$ are given by equations~(\ref{BP}) and~(\ref{BM}), respectively.
\end{prop}
\begin{proof} The formulas~(\ref{AE}) and~(\ref{BK}) define $e_{p,q}$ and $\g_{n,k}$ as linear combinations of 
the numbers $Z(k,r)$, so by comparing the coefficients of $(-2)^{r+2}Z(h+2k-r,r)$ in~(\ref{AE}) and
of $(-2)^{r+2}Z(2p+3q-r,r)$ in~(\ref{BK}) we see that these two formulas follow from the identities
\begin{equation*}
\binom{h-r+2k-3}{2k-2} \,=\, \sum_{0\le 2s<h-r}2^{h-r-2s-2}P_s(h,k)C(k-h+3s,h-2s,r) \qquad(0<r<h) 
\end{equation*}
and
\begin{equation}\label{BS}
2^{q-r-2}\,C(p,q,r) \= \sum_{0\le2s<q-r}Q_s(p,q)\,\binom{2p+3q-r-3}{q-r-1-2s} \qquad(0<r<q), 
\end{equation}
respectively, so we have to show that these two formulas hold if $P_s$ and $Q_s$ are defined by~(\ref{BP}) and~(\ref{BM}),
respectively.  From~(\ref{BM}) (with $t$ replaced by~$x^2$) we have
\begin{align*}
\text{RHS of (\ref{BS})} &\= \text{Coeff}_{x^{q-r-1}}\,\biggl[(1+x)^{2p+3q-r-3}\,\frac{1-9x^2}{{(1+3x^2)}^{p+1}{(1-x^2)}^q}\biggr] \notag\\
 &\= \text{Res}_{x=0}\biggl[\frac{(1-9x^2)(1+x)^{2p+r-3}}{{(1+3x^2)}^{p+1}{(1-x)}^r}\,
          \biggl(\frac{(1+x)^2}{x(1-x)}\biggr)^{q-r}\,dx\biggr] \notag\\
 &\=2^{q-r-2}\,\text{Res}_{z=0}\biggl[\frac{2+z}{{(1-z)}^{p+1}}\,\biggl(\frac{1-\sqrt{1-4z}}2\biggr)^r\,\frac{dz}{z^q}\biggr]\,,
\end{align*}
where in the last equation we have made the substitution $z=\dfrac{2x(1-x)}{(1+x)^2}$, $x=\dfrac{1-\sqrt{1-4z}}{3+\sqrt{1-4z}}\,$.
In view of the definition~(\ref{AF}) of the numbers~$C(p,q,r)$, this is equivalent to~(\ref{BS}). The proof of the inverse identity 
is similar and is left to the reader.\\
\end{proof}

Since the coefficients $P_s(h,k)$ in~(\ref{BO}) are rational, the statement $\g_{n,k}\in\mathcal{R}$ and
hence the (first) proof of Theorem~\ref{Theorem 2} follow from the assertion that the numbers $e_{p,q}$ all belong to~$\mathcal{R}$.  Note that this assertion does not follow directly from the result of Section~\ref{Ssec:TaylorVira} that the numbers $e_{p,q}$ originally defined as the coefficients in~(\ref{AD}) belong to $\mathcal{R}$ by virtue of~(\ref{AC}), because here we also need the coefficients $e_{p,q}$ with $-q<p<0$. But since we proved in Section~\ref{SSec:epq=0} that the latter all belong to~$\mathcal{R}$, and in fact even that they all vanish, this is not a problem.


We now present a second proof of Theorem~\ref{Theorem 2}, which directly relates the generating function 
\begin{equation}\label{defWcl}
W^{({\rm cl})}(X,Y)\;:=\;\frac{1}{X\,(X+Y)\,(Y-X)}\;+\;\sum_{n>k>0}\gamma_{n,k}\,X^{n-k-1}\,Y^{2k-2}
\end{equation}
to the Virasoro function $V^{({\rm cl})}$, avoiding all formulas involving $e_{p,q}$ with $-q<p<0$. We will see in the next section that the different formulas given by these two approaches can be shown to be equivalent by simple combinatorial arguments. 

Let us consider for all $\mu, \nu \geq 0$ the combinations of the special multiple zeta values $Z(k,r)$
\begin{equation*}
g_{\mu,\nu}\=\sum_{s+r=\mu}(-2)^{r}\,\binom{s+\nu}{\nu}\,Z(s+\nu+3,r).
\end{equation*}
Our interest in these numbers here comes from the fact that $\gamma_{n,k}=4\,g_{n-k-1,2k-2}$ and therefore
\begin{equation}\label{prima}
\sum_{n>k>0}\gamma_{n,k}\,X^{n-k-1}\,Y^{2k-2}\=2\,\bigg(\sum_{\mu,\nu\geq 0}\,g_{\mu,\nu}\,X^\mu\,Y^\nu\;+\;\sum_{\mu,\nu\geq 0}\,g_{\mu,\nu}\,X^\mu\,(-Y)^\nu\bigg).
\end{equation}
By the same argument used in the proof of the corollary of Proposition~\ref{Prop2}, if we set $x:=2X$ and $y:=-X-Y$ we have the generating function identity
\begin{equation}\label{seconda}
\sum_{\mu,\nu\geq 0}\,g_{\mu,\nu}\,X^\mu\,Y^\nu\=\frac{1}{x^2}\,\sum_{n\geq 1}\,\binom{x}{n}^2\frac{1}{n+y}.
\end{equation}
Therefore, keeping $x,y$ as above, setting $z=-x-y$ and combining equations~(\ref{prima}) and~(\ref{seconda}) we get
\begin{equation*}
\sum_{n>k>0}\gamma_{n,k}\,X^{n-k-1}\,Y^{2k-2}\=\frac{2}{x^2}\,\sum_{n\geq 1}\,\binom{x}{n}^2\bigg(\frac{1}{n+y}\,+\,\frac{1}{n+z}\bigg).
\end{equation*}
By equations~(\ref{sumrepbeta}) and~(\ref{betaShapiro}), this proves the identity announced in the introduction which explicitly relates the functions~$W^{({\rm cl})}$ and~$V^{({\rm cl})}$, and it concludes our second proof of Theorem~\ref{Theorem 2}:
\begin{prop}\label{ThmMainClosed}
We have
\begin{equation*}
W^{({\rm cl})}(X,Y)\=\frac{V^{({\rm cl})}(2X,-X-Y,Y-X)}{X\,(X+Y)\,(Y-X)}\,.
\end{equation*}
\end{prop}

\subsection{Combinatorial identities}

Alternative and somewhat simpler expressions for the polynomials~$P_s(h,k)$ from~(\ref{BO}) are given by
\begin{equation}\label{BQ}
P_s(h,k) \= c_s(k+3s-h,k) \= (-1)^{h+k}\,3^{3s+1-h}c_{k-1}(h-2s-1,s+1)\,,  
\end{equation}
where~$c_n(a,b)$ denotes the coefficient of~$x^n$ in $(1+4x)^a/(1+x)^{b}$. We leave the proof of these identities as an exercise to the reader.

Moreover, since
\begin{equation*}
-\frac{1}{stu}\,(V^{({\rm cl})}(s,t,u)\,-\,1)\=\frac{1}{T}\,(\Vt^{({\rm cl})}(S,T)\,-\,1)\=\sum_{\substack{p\geq 0\\q\geq 1}}\,e_{p,q}\,S^p\,T^{q-1},
\end{equation*}
from Proposition~\ref{ThmMainClosed} we obtain the formula
\begin{equation}\label{BOnew}
\gamma_{h+k,k}\=\sum_{\substack{2a+c+3d=h-1\\b+c=k-1}}\,(-1)^d\,3^a\,2^{c+d+1}\,\binom{a+b}{a}\binom{c+d}{c}\,e_{a+b,c+d+1}.
\end{equation}
It is another simple combinatorial exercise to show that equation~(\ref{BOnew}) is equivalent to equation~(\ref{BO}) using the second expression for the polynomials~$P_s(h,k)$ from~(\ref{BQ}).

\section{Open string amplitudes and the single-valued projection}\label{Sec:Open}

The open string analogue of the Virasoro function $V^{({\rm cl})}$ is the Veneziano function $V^{({\rm op})}$ on the left-hand side of~(\ref{AA}). It appears in Type~I superstring theory in the genus-zero 4-gluon scattering amplitude, because the latter is essentially computed by the Euler beta function
\begin{equation}\label{RealBeta}
\beta (s,t)\=\int_{0}^1\,x^{s-1}\,(1-x)^{t-1}\,dx
\end{equation}
and it is well-known that, setting $u=-s-t$,
\begin{equation}\label{betaEuler}
\beta(s,t)\=\frac{\Gamma(s)\,\Gamma(t)}{\Gamma(-u)}\=-\frac{u}{st}\,V^{({\rm op})}(s,t,u).
\end{equation}
Similarly to the closed string case, this means that the Taylor expansion at the origin of~$V^{({\rm op})}$ is nothing but the $\alpha'$-expansion of the amplitude, and therefore gives higher-order string corrections to supersymmetric Yang-Mills theories.

\subsection{The Taylor expansion of $V^{({\rm op})}$}\label{Ssec:Veneziano}

Equation~(\ref{ExpVeneziano1}) given in the introduction implies that the Taylor coefficients of the Veneziano function~$V^{({\rm op})}$ are rational linear combinations of products of single zeta values. The next proposition shows that these coefficients have a much simpler expression in terms of the special multiple zeta values
\begin{equation}\label{defH}
H(k,r):=\zeta(\underbrace{1,\cdots ,1}_{r-1},k+1), \quad \quad k,r\geq 1,
\end{equation}
which will be used later to connect genus-zero to genus-one amplitudes. Even though this result is not new and is well-known to experts, we include the proof for completeness.
\begin{prop}\label{LemmaZetas}
The Taylor expansion of~$V^{({\rm op})}(s,t,u)$ with respect to the first two independent variables~$s$ and~$t$ is given by
\begin{equation*}
V^{({\rm op})}(s,t,u)\=1\;+\;\sum_{k,r\geq 1}(-1)^{k+r-1}\,H(k,r)\,s^k\,t^r.
\end{equation*}
\end{prop}
\begin{proof} One can write
\begin{align}\label{intrepH}
H(k,r)&\=\int_{0\leq x_1\leq\cdots \leq x_r\leq y_1\leq\cdots \leq y_k\leq 1}\,\frac{dx_1}{1-x_1}\cdots \frac{dx_r}{1-x_r}\,\frac{dy_1}{y_1}\cdots \frac{dy_k}{y_k} \notag \\
&\=\int_0^1\,\frac{1}{r!}\,\Big(\log\frac{1}{1-y}\Big)^{r}\,\frac{1}{(k-1)!}\,\Big(\log\frac{1}{y}\Big)^{k-1}\,\frac{dy}{y}.
\end{align}
Therefore, for~$s$ and~$t$ small enough we have
\begin{equation*}
\sum_{k,r\geq 1}\,H(k,r)\,s^k\,t^r=s\,\int_0^1\,y^{-s-1}\,\big((1-y)^{-t}\,-\,1\big)\,dy.
\end{equation*}
If $s$ is negative then this is the difference of two convergent integrals. One is the Euler beta integral, the other is elementary,  and we conclude that
\begin{equation*}
\sum_{k,r\geq 1}\,H(k,r)\,s^k\,t^r\=s\,\bigg(\frac{\Gamma(-s)\,\Gamma(1-t)}{\Gamma(1-s-t)}\,+\,\frac{1}{s}\bigg)\=1\,-\,\frac{\Gamma(1-s)\,\Gamma(1-t)}{\Gamma(1-s-t)}.
\end{equation*}
\end{proof}

We conclude our discussion of the Taylor coefficients of~$V^{({\rm op})}$ by expressing the \MZVs\ $H(k,r)$ in terms of special multiple series known as ``Mordell-Tornheim sums'', because the latter will turn up later in the proof of Theorem~\ref{PropOp}.
\begin{lemma}\label{LemmaTornheim}
For any $k,r\geq 1$ one has
\begin{equation*}
H(k,r)\=\frac{1}{r!}\,\sum_{n_1,\ldots ,n_r\geq 1}\,\frac{1}{n_1\cdots n_r\,(n_1+\cdots +n_r)^k}.
\end{equation*}
\end{lemma}
\begin{proof} The statement follows by comparing~(\ref{intrepH}) with the integral representation
\begin{align*}
\sum_{n_1,\ldots ,n_r\geq 1}\,\frac{1}{n_1\cdots n_r\,(n_1+\cdots +n_r)^k}&\=\int_{0\leq x_1,\ldots ,x_r\leq y_{1}\leq\cdots \leq y_{k}\leq 1}\,\frac{dx_1}{1-x_1}\cdots \frac{dx_r}{1-x_r}\,\frac{dy_1}{y_1}\cdots \frac{dy_k}{y_k} \notag \\
&\=\int_0^1\,\,\Big(\log\frac{1}{1-y}\Big)^{r}\,\frac{1}{(k-1)!}\,\Big(\log\frac{1}{y}\Big)^{k-1}\,\frac{dy}{y}.
\end{align*}
Alternatively, it follows from the stronger identity
\begin{equation*}
\frac{1}{z_1\cdots z_r\,(z_1+\cdots +z_r)^k}\=\sum_{\sigma\in\mathfrak{S}_r}\,\frac{1}{z_{\sigma(1)}\cdots (z_{\sigma(1)}+\cdots +z_{\sigma(r-1)})\,(z_{\sigma(1)}+\cdots +z_{\sigma(r)})^{k+1}},
\end{equation*}
where~$\mathfrak{S}_r$ is the symmetric group on~$r$ letters, which is an identity of rational functions and can be proven using partial fractioning and induction on~$r$. \\
\end{proof}

\subsection{Two-point holomorphic graph functions}\label{Ssec:hgf}

\emph{Holomorphic graph functions} are open string analogues of the modular graph functions $D_\l(\tau)$ from Section~\ref{Sec:MGF}. They were introduced and related to genus-one open superstring amplitudes\footnote{More precisely, holomorphic graph functions occur in a symmetrized version of the annulus-worldsheet contribution to the scattering of 4-gluons in Type~I superstring theory.} in~\cite{BSZ}. Here we are only interested in ``two-vertex B-cycle holomorphic graph functions''.

For $\tau\in i\mathbb{R}^+$ and $\l\geq 0$ we consider the integral\footnote{If $\l= 1$ then the integral diverges and we regularize it following~\cite{BSZ} by considering a tangential base point. Our normalization of the propagator $P_B$ is picked in such a way that $B_1(\tau)=0$.}
\begin{equation}\label{BholGF}
B_{\l}(\tau)\;:=\;\int_{\mathbb{R}\tau/\mathbb{Z}\tau}\,P_B(z,\tau)^{\l}\,\frac{dz}{\tau},
\end{equation}
where the ``B-cycle open string propagator'' $P_B(z,\tau)$ is given by
\begin{equation*}
P_B(z,\tau)=-\log\bigg|\frac{\theta_1(z,\tau)}{\eta(\tau)}\bigg|-\frac{i\pi}{\tau}\bigg(z^2+\frac{1}{6}\bigg)
\end{equation*}
and can be seen (up to an additive constant) as half of the restriction of the Green's function on the torus $G(z,\tau)$ to the B-cycle $\mathbb{R}\tau/\mathbb{Z}\tau$. The integral in~(\ref{BholGF}) is well-defined for $\tau\in i\mathbb{R}^+$ because then $P_B(z+\tau,\tau)=P_B(z,\tau)$. The (unique) holomorphic extension of $B_{\l}(\tau)$ to the whole complex upper half-plane $\mathbb{H}$ coincides with the (B-cycle) holomorphic graph function $B(\mathcal{G},\tau)$ defined in~\cite{BSZ} for a graph~$\mathcal{G}$ with two vertices and~$\l$ edges. For this reason we keep writing~$\tau$ instead of using the real variable $t=\tau/i$ and by abuse of terminology we refer to~$B_{\l}(\tau)$ as ``holomorphic'' graph functions.

It is known that each $B_\l(\tau)$ can be expressed in terms of elliptic multiple zeta values, and therefore as (special) linear combinations of iterated integrals of holomorphic Eisenstein series~\cite{BMMS, Enriquez}. For instance, it was shown in~\cite{BSZ} that $B_2(\tau)$ is essentially given by the Eichler integral of $G_4(\tau)$ and $B_3(\tau)$ by the Eichler integral of $G_6(\tau)$, while for $\l\geq 4$ one needs Manin-Brown's theory of iterated Eichler integrals~\cite{ManinItInt, MMV}. Similarly to the closed string case, in the present work we will only focus on the behaviour at infinity of holomorphic graph functions. We start by recalling the  following:
\begin{prop}[Br\"odel, Schlotterer, Zerbini \cite{BSZ}]
For any $\l\geq 0$ we have
\begin{equation*}
B_{\l}(\tau)\=b_\l(T)\,+\,O(e^{-T})
\end{equation*}
for $T=\pi\tau/i\rightarrow +\infty$ and a Laurent polynomial $b_\l(T)$ whose coefficients are rational linear combinations of multiple zeta values.
\end{prop}

The method originally used to prove this proposition does not provide an explicit formula for the polynomials~$b_\l(T)$, but the numerical calculations performed up to~$\l=6$ seemed to indicate that their coefficients should actually belong the smaller $\mathbb{Q}$-ring generated by single zeta values. Our first result on holomorphic graph function is a proof of this conjecture, which is obtained by writing an explicit formula for $b_\l(T)$.
\begin{thm}\label{PropOp}
The coefficients of the Laurent polynomials~$b_\l(\tau)$ are rational linear combinations of products of single zeta values. They are explicitly given in terms of the special \MZVs\ ~$H(k,r)$ from~(\ref{defH}) by the formula
\begin{align}\label{Formula2pointopen}
b_\l(T)&\=\sum_{\substack{a+b+c+d=\l\\a,b,c,d\geq 0}}\frac{\l!}{a!\,b!\,c!\,d!}\frac{(-1)^{b+d}\,\zeta(2)^d}{6^c\,(2a+b+1)}\,T^{\l-2d} \notag\\
&\;+\;\sum_{\substack{a+b+c+d+e=\l\\a,b,c,d\geq 0,\,e\geq 1}}\,\frac{\l!}{a!\,b!\,c!\,d!}\frac{(-1)^{b+d}\,(2a+b)!\,\zeta(2)^d}{2^{2a+b}\,6^c}\,H(2a+b+1,e)\,T^{c-a-d-1}.
\end{align}
\end{thm}
\begin{proof} Fixing the fundamental domain $[-\tau/2,\tau/2]$, using the fact that $P_B(-z,\tau)=P_B(z,\tau)$ and rescaling the integration variable, we get
\begin{equation*}
B_\l(\tau)\=2\int_{0}^{1/2}P_B(z\tau ,\tau )^{\l}\,dz.
\end{equation*}
For $z\in [0,1]$ and $\tau\in i\mathbb{R}^+$ Jacobi's triple product formula for $\theta_1$ gives
\begin{equation*}
P_B(z\tau,\tau)\=R(z,\tau)\,+\,S(z,\tau),
\end{equation*}
where we define for $\tilde{u}=\exp(2\pi i\tau z)$ and $T=\pi \tau/i$
\begin{equation}\label{L}
R(z,\tau)\=T\,\bigg(z^2-z+\frac{1}{6}\bigg)\,-\,\frac{\zeta(2)}{T},
\end{equation}
\begin{equation}\label{S}
S(z,\tau)\=\sum_{m\geq 1}\,\frac{\tilde{u}^m}{m}\,+\,\sum_{n,m\geq 1}\,\frac{\tilde{u}^m\,q^{nm}}{m}\,+\,\sum_{n,m\geq 1}\frac{\tilde{u}^{-m}\,q^{nm}}{m}.
\end{equation}
Therefore
\begin{equation*}
B_\l(\tau)\=2\,\int_0^{1/2}\,(R(z,\tau)\,+\,S(z,\tau))^\l\,dz\=2\,\sum_{\substack{r+s=\l\\r,s\geq 0}}\,\frac{\l!}{r!\,s!}\,\int_0^{1/2}\,R(z,\tau)^r\,S(z,\tau)^s\,dz.
\end{equation*}
If $s=0$, by the invariance of the second Bernoulli polynomial $z^2-z+1/6$ under $z\rightarrow 1-z$ we get
\begin{align*}
2\,\int_0^{1/2}\,R(z,\tau)^\l\,dz\=\int_0^{1}\,R(z,\tau)^\l\,dz&\=T^\l\,\int_0^1\,\bigg(z^2\,-\,z\,+\,\frac{1}{6}\,-\,\frac{\zeta(2)}{T^2}\bigg)^\l\,dz\notag\\
&\=\sum_{\substack{a+b+c+d=l\\a,b,c,d\geq 0}}\,\frac{\l!}{a!\,b!\,c!\,d!}\frac{(-1)^{b+d}\,\zeta(2)^d}{6^c\,(2a+b+1)}\,T^{\l-2d},
\end{align*}
which is the first term appearing in~(\ref{Formula2pointopen}). 

Let us now suppose that $s\geq 1$. Using the explicit expressions~(\ref{L}) and~(\ref{S}), we write
\begin{equation*}
\int_0^{1/2}\,(R(z,\tau)\,+\,S(z,\tau))^\l\,dz\=\sum_{\substack{a+b+c+d=r\\e+f+g=s\\a,b,c,d,e,f,g\geq 0}}\,\frac{r!\,s!}{a!\,b!\,c!\,d!\,e!\,f!\,g!}\frac{(-1)^{b+d}\,\zeta(2)^d}{6^c}\,T^{a+b+c-d}\,I_{a,b,e,f,g}(T),
\end{equation*}
where
\begin{equation}\label{crazyint}
I_{a,b,e,f,g}(T)\=\int_0^{1/2}\,z^{2a+b}\,\bigg(\sum_{m\geq 1}\frac{e^{-2mTz}}{m}\bigg)^e\,\bigg(\sum_{k,h\geq 1}\frac{e^{-2kT(h+z)}}{k}\bigg)^f\,\bigg(\sum_{u,v\geq 1}\frac{e^{-2uT(v-z)}}{u}\bigg)^g\,dz.
\end{equation}
We need to take into account only the contributions from~(\ref{crazyint}) which are not exponentially suppressed as $T\rightarrow+\infty$, so we can immediately set $f=g=0$. A straightforward computation gives
\begin{equation*}
I_{a,b,e,0,0}(T)\=\frac{(2a+b)!}{(2T)^{2a+b+1}}\,\sum_{m_1,\ldots ,m_e\geq 1}\frac{1}{m_1\cdots m_e\,(m_1+\cdots +m_e)^{2a+b+1}}\,+\,O(e^{-T}),
\end{equation*}
and so by Lemma~\ref{LemmaTornheim} we obtain also the second term of formula~(\ref{Formula2pointopen}). This concludes the proof of the theorem, because we know by Proposition~\ref{LemmaZetas} that each~$H(k,r)$ is a polynomial in single zeta values with rational coefficients.\\
\end{proof}

Similarly to the closed string case, considering the generating function of the polynomials~$b_\l(T)$ leads to simpler Laurent polynomials in $T$:
\begin{cor}
We have
\begin{equation}\label{genserbl}
\sum_{\l\geq 0}\,b_\l(T)\,\frac{s^\l}{\l!}\=e^{sT/6}\,e^{-\zeta(2)s/T}\,\sum_{n\geq 0}\bigg(\frac{T^n}{(2n+1)!!}\,+\,\sum_{k=1}^n\frac{(-1)^{k-1}\,(2k-3)!!\,\eta_{n,k}}{T^k}\bigg)\bigg(-\frac{s}{2}\bigg)^n,
\end{equation}
where
\begin{equation*}
\eta_{n,k}\;:=\;\sum_{r=-1}^{n-k-1}\,(-2)^{r+2}\,\binom{k+n-r-3}{2k-2}\,H(k+n-r-2,r+2).
\end{equation*}
\end{cor}

\subsection{Relating $d_\l$ to $b_\l$ via Brown's single-valued projection}\label{Ssec:ProofEsvConj} 

Following the work of Brown~\cite{BrownSVMZV} and assuming standard conjectures on the structure of the ring of multiple zeta values $\mathcal{Z}$, one can define a map $\mbox{sv}:\mathcal{Z}\rightarrow\mathcal{Z}$ that sends multiple zeta values $\zeta(\textbf{k})$ to \emph{single-valued multiple zeta values}\footnote{We do not need here to recall Brown's general construction of single-valued periods, nor its specialization to the $\mathbb{Q}$-algebra $\mathcal{Z}$ of multiple zeta values. We just mention that single-valued multiple zeta values are special values of single-valued multiple polylogarithms, as opposed to the fact that multiple zeta values are special values of (multi-valued) classical multiple polylogarithms, and this explains the surprising term ``single-valued'' associated to a set of numbers.} $\zeta^{\rm sv}(\textbf{k})$, which generate a (conjecturally) much smaller ring $\mathcal{Z}^{\rm sv}\subset \mathcal{Z}$. We will call $\mbox{sv}$ the \emph{single-valued projection}. In particular, as already mentioned in~(\ref{singlevaluedZetas}), $\zeta^{\rm sv}(2k)=0$ and $\zeta^{\rm sv}(2k+1)=2\zeta(2k+1)$, therefore by looking at the power series expansions~(\ref{ExpVeneziano1}) and~(\ref{ExpVirasoro1}) it is immediately clear that one can obtain $V^{({\rm cl})}(s,t,u)$ from $V^{({\rm op})}(s,t,u)$ by applying the single-valued projection coefficientwise. This is the 4-point case of a recently proved theorem~\cite{BrownDupont, SchlSchn}, whose statement was conjectured by Stieberger~\cite{Stieb2014}, which predicts that, at genus zero, $n$-point closed string amplitudes are the image of $n$-point open string amplitudes under the single-valued projection (applied term-by-term to the $\alpha'$-expansion).

Analogous statements are expected to hold true for higher-genus string amplitudes. The main result of this section, which we called Theorem B in the introduction and whose statement we repeat here, proves the 2-point case of a general conjecture about single-valued projections of string amplitudes at genus one that was stated in~\cite{BSZ}.
\begin{thm}\label{thmsvg1}
For every integer $\l\geq 0$ the formal extension by linearity of the single-valued projection\footnote{If we keep the original variables $T=\pi\tau/i=-\log(q)/2$ and $Y=2\pi\text{Im}(\tau)=-\log|q|$, we can interpret the formal operation of mapping $T$ to $Y$ as yet another single-valued projection, because the single-valued image of $\log(x)$ is $\log(x)+\overline{\log(x)}$.} to Laurent polynomials gives
\begin{equation*}
{\rm sv}\big(b_\l(X)\big)\=d_\l(X).
\end{equation*}
\end{thm}
For example, one can look at the first non-trivial open string Laurent polynomials
\begin{align*}
b_2(X)&\=
\frac{X^2}{180} + \frac{\zeta(2)}{3} + \frac{\zeta(3)}{X} - \frac{3\, \zeta(4)}{2\, X^2}\,, \notag \\
b_3(X)&\=\frac{X^3}{3780}+\frac{\zeta(2)\,X}{15} + \frac{\zeta(3)}{2} + \frac{19\,\zeta(4)}{4\,X} + \frac{3\, \zeta(5)}{2\, X^2}-\frac{6\,\zeta(3)\,\zeta(2)}{X^2}+\frac{8\,\zeta(6)}{X^3}\,,
\end{align*}
and see how they are mapped to the corresponding closed string Laurent polynomials
\begin{align*}
d_2(X)&\=\frac{X^2}{180}+\frac{2\zeta(3)}{X}\,,\notag \\
d_3(X)&\=\frac{X^3}{3780}+\zeta(3)+\frac{3\, \zeta(5)}{X^2}\,.
\end{align*}

Comparing the generating series~(\ref{genserbl}) and~(\ref{BJ}) of the Laurent polynomials $b_\l$ and $d_\l$ and remembering that $\zeta^{\rm sv}(2)=0$, it is immediately clear that proving Theorem~\ref{thmsvg1} is equivalent to proving that for all $n\geq k>0$ we have ${\rm sv}(\eta_{n,k})=\gamma_{n,k}$.\footnote{We define $\gamma_{n,k}:=0$ for $n=k$.} Therefore if we introduce the generating series
\begin{equation*}
W^{({\rm op})}(X,Y)\;:=\;\frac{1}{X\,(X+Y)\,(Y-X)}\;+\;\sum_{n\geq k>0}\eta_{n,k}\,X^{n-k-1}\,Y^{2k-2},
\end{equation*}
we need to prove that $\mbox{sv}\big(W^{({\rm op})}(X,Y)\big)=W^{({\rm cl})}(X,Y)$, with $W^{({\rm cl})}(X,Y)$ as in~(\ref{defWcl}).
\begin{prop}\label{propWop}
We have the generating series identity
\begin{equation}\label{bas}
W^{({\rm op})}(X,Y)\=\bigg(\frac{\sin\big(\pi(X+Y)\big)}{\sin\big(\pi(Y-X)\big)}\,-\,1\bigg)\,\frac{V^{({\rm op})}(2X,-X-Y,Y-X)}{\,2\,X^2\,(X+Y)}\,.
\end{equation}
\end{prop}
\begin{proof} Let us define, for all $\mu,\nu\geq 0$,
\begin{equation*}
h_{\mu,\nu}\;:=\;\sum_{s=0}^\mu\,(-2)^{\mu-s+1}\binom{\nu+s}{\nu}\,H(\nu+s+1,\mu-s+1).
\end{equation*}
Then by Proposition~\ref{LemmaZetas} we get
\begin{align*}
&\sum_{\mu,\nu\geq 0}h_{\mu,\nu}\,X^\mu\, Y^\nu\=\sum_{r,s,\nu\geq 0}(-2)^{r+1}\binom{\nu+s}{\nu}\,H(\nu+s+1,r+1)\,X^{r+s}\,Y^{\nu} \notag \\
&\=-2\sum_{k,r\geq 0}H(k+1,r+1)(-2X)^{r}(X+Y)^{k}=\frac{1}{X(X+Y)}\bigg(1-\frac{\Gamma(1+2X)\Gamma(1-X-Y)}{\Gamma(1+X-Y)}\bigg).
\end{align*}
Since $h_{n-k,2k-2}=\eta_{n,k}$ for all $n\geq k>0$, we have
\begin{align*}
\sum_{n\geq k>0}\eta_{n,k}\,X^{n-k-1}\,&Y^{2k-2}\=\frac{1}{2X}\,\bigg(\sum_{\mu,\nu\geq 0}h_{\mu,\nu}\,X^\mu\, Y^\nu\,+\,\sum_{\mu,\nu\geq 0}h_{\mu,\nu}\,X^\mu\, (-Y)^\nu\bigg) \notag \\
&\=\frac{1}{X}\,\Bigg(\frac{1}{(X+Y)\,(X-Y)}\,+\,\frac{\Gamma(2X)\,\Gamma(-X-Y)}{\Gamma(1+X-Y)}\,+\,\frac{\Gamma(2X)\,\Gamma(-X+Y)}{\Gamma(1+X+Y)}\Bigg) \notag \\
&\= \frac{1}{X}\,\Bigg(\frac{1}{(X+Y)\,(X-Y)}+\frac{\Gamma(2X)\,\Gamma(-X-Y)}{\Gamma(1+X-Y)}\,\bigg(1\,-\,\frac{\sin\big(\pi(X+Y)\big)}{\sin\big(\pi(Y-X)\big)}\bigg)\Bigg),
\end{align*}
where in the last line we have used the reflection identity $\Gamma(z)\Gamma(1-z)=\pi/\sin(\pi z)$ of $\Gamma$.\\
\end{proof}
\begin{proof}[Proof of Theorem~\ref{thmsvg1}]
We need to prove that $\mbox{sv}\big(W^{({\rm op})}(X,Y)\big)=W^{({\rm cl})}(X,Y)$. First note that
\begin{equation*}
{\rm sv}\Bigg(\frac{\sin\big(\pi(X+Y)\big)}{\sin\big(\pi(Y-X)\big)}\,-\,1\Bigg)\=\frac{2X}{Y-X}.
\end{equation*}
Indeed, all rational coefficients of the power series expansion of the quotient of sines are multiplied by a non-negative even power of $\pi$, so by $\mathbb{Q}$-linearity and Euler's theorem ${\rm sv}(\pi^{2k})={\rm sv}(\zeta(2k))=0$ for all $k\geq 1$. The term $2X/(Y-X)$ is the only one in the expansion whose coefficient is just a rational number. Moreover, we have already observed that $\mbox{sv}\big(V^{({\rm op})}(s,t,u)\big)=V^{({\rm cl})}(s,t,u)$. The statement follows from Propositions~\ref{propWop} and~\ref{ThmMainClosed}.\\
\end{proof}

\subsection{Relating $d_\l$ to $b_\l$ via the KLT formula}\label{Ssec:KLT}

We have seen that the classical and complex beta functions, defined by~(\ref{RealBeta}) and~(\ref{ComplexBeta}), respectively, can be related using Brown's single-valued projection coefficientwise in the $\alpha'$-expansion. By equations~(\ref{betaShapiro}) and~(\ref{betaEuler}), another obvious relation is given (setting $u=-s-t$) by
\begin{equation}\label{KLT1}
\beta_{\mathbb{C}}(s,t)\=\frac{u}{st}\,\beta(s,t)\,\beta(-s,-t)^{-1}.
\end{equation} 
Using the reflection property of $\Gamma$, this identity can be rewritten as
\begin{equation}\label{KLT2}
\beta_{\mathbb{C}}(s,t)\=-\frac{\sin(\pi s)\sin(\pi t)}{\pi\,\sin(\pi u)}\beta(s,t)^2,
\end{equation}
which is the 4-point case of a general formula of geometric origin, found by Kawai, Lewellen and Tye~\cite{KLT} and known as the \emph{KLT formula}, which allows to write genus-zero closed string amplitudes as ``double copies'' of genus-zero open string amplitudes.\footnote{One can see the single-valued projection on multiple zeta values and the KLT formula as two different instances of a ``single-valued integration pairing'' for the de Rham cohomology, as discussed in~\cite{BrownDupont}.}


By Propositions~\ref{ThmMainClosed} and~\ref{propWop} we can use~(\ref{KLT1}) and~(\ref{KLT2}) to obtain the expressions
\begin{equation}\label{KLT3}
W^{({\rm cl})}(X,Y)\=\frac{W^{({\rm op})}(X,Y)\,W^{({\rm op})}(-X,Y)^{-1}}{X\,(X+Y)\,(X-Y)}
\end{equation}
and
\begin{equation}\label{KLT4}
W^{({\rm cl})}(X,Y)\=\frac{2X^2}{\pi}\,\frac{\sin\big(\pi(2X)\big)\,\sin\big(\pi(X+Y)\big)\,\sin\big(\pi(Y-X)\big)}{\big(\sin\big(\pi(X+Y)\big)\,+\,\sin\big(\pi(X-Y)\big)\big)^2}\,W^{({\rm op})}(X,Y)^2,
\end{equation}
respectively.

These formulas yield (equivalent) non-trivial relations between the special combinations of \MZVs\ $\eta_{n,k}$ and $\gamma_{n,k}$. Since these are the coefficients of (formal) 2-point amplitudes at genus one, it would be interesting to know whether equations~(\ref{KLT3}) and~(\ref{KLT4}) have a geometric origin and can be interpreted as ``genus-one double-copy relations''.

\section*{Acknowledgements} We would like to thank J.~Br\"odel and O.~Schlotterer for useful comments. The research of F.~Zerbini was supported by a French public grant as part of the Investissement d'avenir project,
reference ANR-11-LABX-0056-LMH, LabEx LMH, and by the People Programme (Marie Curie Actions) of the European Union's Seventh Framework Programme
(FP7/2007-2013) under REA grant agreement n. PCOFUND-GA-2013-609102, through the PRESTIGE programme coordinated by Campus France.

\section*{Appendix: Green's functions and closed string amplitudes}

In closed string theory one has to integrate over the moduli spaces $\mathfrak{M}_{g,n}$ of 
Riemann surfaces of genus~$g$ with~$n$ marked points, where $2-2g-n<0$.  In particular, one has to
introduce natural measures on these moduli spaces and find natural functions to integrate.  One way,
which arises naturally in the calculation of certain amplitudes, will be sketched here very briefly.
%
%
To define a function on~$\mathfrak{M}_{g,n}$, we must associate a number in a coordinate-independent way to 
each pair $(C,\textbf{z})=(C;z_1,\dots,z_n)$ consisting of a curve~$C$ and~$n$ marked points~$z_i\in C$.  This can be done
using the theory of Green's functions on Riemann surfaces.  We recall that for each metric~$\nu$ on $C$ compatible with its
conformal structure, one has a Green's function $G_\nu$ which is a symmetric and real-analytic function on the complement
of the diagonal in $C\times C$, with a logarithmic singularity near the diagonal (more precisely, with
$G_\nu(z,z_0)=\log|t_0(z)|^2+\O(1)\,$ as $z\to z_0\in C$, where $t_0$ is a local coordinate near~$z_0$
vanishing at~$z_0$), and such that $z\mapsto G_\nu(z,z_1)-G_\nu(z,z_2)$ is harmonic on $C\smallsetminus\{z_1,z_2\}$
for any $z_1,\,z_2\in C$.  The Green's function depends on the metric only by the addition of functions
depending on its two arguments separately, so if we have two divisors $D=\sum_in_i(z_i)$ and $D'=\sum_jn'_j(z'_j)$
of degree~0 on~$C$, then the ``height pairing" $\langle D,D'\rangle=\sum_{i,j}n_in'_jG_\nu(z_i,z'_j)$ is independent of the
metric~$\nu$, is real-analytic as long as no two $z_i$'s coincide, and is harmonic (away from its singularities)
with respect to each variable~$z_i$. In particular, the 4-variable function
\begin{equation*}
G(z_1,z_2;z_3,z_4) \= \langle (z_1)-(z_2),\,(z_3)-(z_4)\rangle \= G_\nu(z_1,z_3)\,-\,G_\nu(z_2,z_3)\,-\,G_\nu(z_1,z_4)\,+\,G_\nu(z_2,z_4)
\end{equation*}
is independent of the metric, real-analytic on $\{(z_1,\dots,z_4)\in C^4\mid\{z_1,z_2\}\cap\{z_3,z_4\}=\emptyset\}$,
and harmonic in each variable.  One also considers the exponentiated function
\begin{equation*}
H(z_1,z_2;z_3,z_4) \= e^{G(z_1,z_2;z_3,z_4)} \= \frac{H_\nu(z_1,z_3)H_\nu(z_2,z_4)}{H_\nu(z_2,z_3)H_\nu(z_1,z_4)} 
  \qquad \bigl(\,H_\nu(z,z')\=e^{G_\nu(z,z')}\,\bigr)\,,
\end{equation*}
which is again independent of the metric. In the case~$g=0$, if we identify $C$ with the Riemann sphere $\mathbb{P}^1_{\mathbb{C}}$, then 
$H(z_1,z_2;z_3,z_4)=\Bigl|\dfrac{(z_1-z_3)(z_2-z_4)}{(z_2-z_3)(z_1-z_4)}\Bigr|^2$ is just the absolute value
squared of the cross-ratio.  A more general combination of Green's functions, defined directly on tuples $(C,\bz)$ 
as above, is given by the sum $\textbf{G}(C,\bz;\bs)=\sum_{1\le i<j\le n}s_{ij}G_\nu(z_i,z_j)$, where $\bs=(s_1,\dots,s_n)$
belongs to the $\h n(n-3)\,$-\,dimensional vector space $\X_n$ of $n\times n$ symmetric matrices of (real or complex) numbers with 
vanishing diagonal entries and vanishing row sums.  (The entries of~$\bs$ correspond physically to the
scalar products of the momenta of~$n$ massless particles, which are the Mandelstam variables mentioned in the introduction.)
The fact that $\bs$ has vanishing row (and column) sums implies
that $\textbf{G}(C,\bz;\bs)$ depends only on the complex structure on~$C$ and not on the metric chosen to compute the
individual Green's functions.  Again we also have the corresponding exponentiated function 
$\textbf{H}(C,\bz;\bs)=e^{\textbf{G}(C,\bz;\bs)}=\prod_{1\le i<j\le n}H_\nu(z_i,z_j)^{s_{ij}}$.  The amplitude of interest is then the 
integral (with respect to a suitable metric) of $\textbf{H}(C,\bz;\bs)$ over $[C,\bz]\in\mathfrak{M}_{g,n}$, and is a function of
the Mandelstam variable~$\bs$.  If $n=4$, then $\X_n$ corresponds to the $\X$ defined at the beginning of the article 
by assigning to a matrix $\bs\in\X_4$ its first row $(0\,s\,t\,u)$, which is easily seen to determine the rest, 
and the integral in question becomes a function of the variables~$S$ and~$T$ defined by~(\ref{AB}). For $g=0$, this
integral is essentially equal to the Virasoro function $V^{({\rm cl})}(s,t,u)$ studied in the first part of the paper.

In the case $g=1$, there is a canonical choice of Green's function by taking $\nu$ to be the flat metric, and the
corresponding function~$G_\nu(z_1,z_2)$ depends only on $z=z_1-z_2$ and agrees with the function 
$G(z,\t)$ given in~(\ref{DefGreenfct}) in terms of the Jacobi theta function, where $C=\C/\Lambda_\t$.
In this case we do not have to pass to divisors of degree~0 or use the set~$\X_n$ in order to define the amplitudes, but 
already get an interesting 2-point function, or even 1-point function if we use the translation invariance to place
one of the points at the origin of the elliptic curve.  This leads to the expression $\,\text{Av}_z\bigl(H(z,\t)^s\bigr)
=\sum_\l D_\l(\t)s^\l/\l!$ studied in Section~\ref{Sec:MGF}.

\bigskip

\bibliographystyle{plain}

\end{document}